\documentclass[reqno,10pt,a4paper]{amsart}
\usepackage[utf8]{inputenc}

\theoremstyle{definition}

\theoremstyle{remark}

\numberwithin{equation}{section}

\usepackage{amssymb,eucal,mathrsfs,array,setspace,geometry,enumitem,cite,tensor,
amsmath,wrapfig,amscd,mathptmx,graphicx,bm,multirow,multicol,adjustbox}
\usepackage[dvipsnames]{xcolor}
\usepackage{tikz}
\usetikzlibrary{knots,calc}
\usetikzlibrary{decorations.markings}
\usepackage[centertableaux]{ytableau}
\usepackage{txfonts}            % arXiv does not support newtexmath  yet. Falling back on txfonts
\usepackage{mathtools} % needed for cases* and the bra/ket constructions
\usepackage{stmaryrd} % adds llbracket etc...

\usepackage{caption}
\usepackage{subcaption}
\usepackage{xcolor}

\usetikzlibrary{positioning}
\usetikzlibrary{decorations.markings}

\usepackage[colorlinks=true,citecolor=red,linkcolor=blue]{hyperref} % load _before_ cleveref!
\usepackage[capitalise,noabbrev]{cleveref}

\usepackage{tikz-cd}
\usetikzlibrary{decorations.pathmorphing}

\usetikzlibrary{knots}

\DeclareSymbolFont{largesymbols}{OMX}{zplm}{m}{n} % Replaces summation/product symbols in txmath by the palatino ones...

\setitemize{leftmargin=*}   % Removes margin from itemized lists (enumitem package)
\setenumerate{leftmargin=*} % Removes margin for enumerate lists (enumitem package)
\setlist[enumerate,1]{label=\textup{(\arabic*)}, % and makes labels upright numbers
  ref=\arabic*} 
\setlist[enumerate,2]{label = \textup{(\roman*)},
  ref = \theenumi.\roman*}

\let\originalleft\left     % removes spurious spacing around \left and \right brackets
\let\originalright\right
\renewcommand{\left}{\mathopen{}\mathclose\bgroup\originalleft}
\renewcommand{\right}{\aftergroup\egroup\originalright}

\newcolumntype{C}{>{$}c<{$}} %Defines math mode in tabular (array package)

\numberwithin{equation}{section}
\allowdisplaybreaks

\renewcommand{\ge}{\geq}%slant}
\renewcommand{\le}{\leq}%slant}

\DeclarePairedDelimiter{\brac}{\lparen}{\rparen} % use \brac for (...) and \brac* to automatically scale the ( and )
\DeclarePairedDelimiter{\sqbrac}{\lbrack}{\rbrack} % use \sqbrac[\big] for \bigl(...\bigr) etc...
\DeclarePairedDelimiter{\set}{\lbrace}{\rbrace}
\newcommand{\st}{\mspace{5mu} \vert \mspace{5mu}} % "such that" in sets
\DeclarePairedDelimiter{\abs}{\lvert}{\rvert}

\DeclarePairedDelimiter{\powser}{\llbracket}{\rrbracket} % [[ ... ]] using stmaryrd

\DeclarePairedDelimiterX{\comm}[2]{\lbrack}{\rbrack}{#1 , #2}  % commutators
\DeclarePairedDelimiterX{\acomm}[2]{\lbrace}{\rbrace}{#1 , #2} % anticommutators
\DeclarePairedDelimiterX{\super}[2]{\lparen}{\rparen}{#1 \delimsize\vert \mathopen{} #2} % for super args (m|n)

\DeclareMathOperator{\diag}{diag}

\newcommand{\wun}{\mathbf{1}}  % the unit of all sorts of things

\DeclareMathOperator{\Res}{Res}

\newcommand{\lra}{\longrightarrow}
\DeclareMathOperator{\Hom}{Hom}

\newcommand{\fld}[1]{\mathbb{#1}}    % for fields and related things
\newcommand{\alg}[1]{\mathfrak{#1}}  % for Lie algebras
\newcommand{\grp}[1]{\mathsf{#1}}    % for groups
\newcommand{\categ}[1]{\mathscr{#1}} % categories (requires mathrsfs package)

\newcommand{\ZZ}{\fld{Z}}
\newcommand{\NN}{\fld{N}}
\newcommand{\QQ}{\fld{Q}}
\newcommand{\RR}{\fld{R}}
\newcommand{\CC}{\fld{C}}
\newcommand{\HH}{\fld{H}}

\newcommand{\SLTZ}{\grp{SL\brac*{2,\ZZ}}}

\newcommand{\SLTZp}[1]{\grp{SL}\brac*{2,\ZZ_{#1}}}
\newcommand{\GL}[1]{\grp{GL}\brac*{#1,\CC}}
\newcommand{\Bthree}{\grp{B}_3}

\newcommand{\SLA}[2]{\alg{#1} \brac*{#2}}                 % Lie algebras like sl(2)
\DeclareMathOperator{\wt}{wt}                                               % conformal weight
\newcommand{\catC}{\categ{C}}           

\newcommand{\rep}[1]{\mathsf{Rep}\ \brac*{#1}}

\DeclarePairedDelimiter{\ket}{\lvert}{\rangle}
\DeclarePairedDelimiterX{\braket}[2]{\langle}{\rangle}{#1 \delimsize\vert \mathopen{} #2}
\DeclarePairedDelimiterX{\bracket}[3]{\langle}{\rangle}{#1 \delimsize\vert \mathopen{} #2 \delimsize\vert \mathopen{} #3}
\DeclareMathOperator{\tr}{tr}

\newcommand{\chmap}{\mathrm{ch}}
\newcommand{\Gr}[1]{\sqbrac[\big]{#1}}          % element of a Grothendieck group/ring
\newcommand{\ch}[1]{\chmap \Gr{#1}}             % characters
\newcommand{\modS}{\mathsf{S}}                        % modular S-matrix
\newcommand{\modT}{\mathsf{T}}                        % modular T-matrix
\newcommand{\Rmat}{\mathsf{R}}                        
\newcommand{\Fmat}{\mathsf{F}}
\newcommand{\Gmat}{\mathsf{G}}

\newcommand{\fuse}{\mathbin{\boxtimes}}                                            % fusion
\newcommand{\iop}[1]{\mathcal{Y}_{#1}}

\newcommand{\ityp}[3]{\binom{#3}{#1,#2}}
\newcommand{\ispc}[3]{\operatorname{I}\binom{#3}{#1,#2}}

\newcommand{\cft}{conformal field theory}
\newcommand{\Cft}{Conformal field theory}

\newcommand{\voa}{vertex operator algebra}

\newcommand{\cfin}[1]{\(C_{#1}\)-cofinite}
\newcommand{\ctwo}{\cfin{2}}
\newcommand{\mtc}{modular tensor category}

\renewcommand{\epsilon}{\varepsilon}
\newcommand{\vvmf}{vector-valued modular form}

\theoremstyle{plain}
\newtheorem{thm}{Theorem}[section]
\newtheorem{prop}[thm]{Proposition}
\newtheorem{lem}[thm]{Lemma}
\newtheorem{cor}[thm]{Corollary}
\newtheorem*{thm*}{Theorem}

\theoremstyle{definition} % Non-italicised text

\newtheorem{defn}[thm]{Definition}
\newtheorem*{rmk}{Remark}

\Crefname{thm}{Theorem}{Theorems}
\Crefname{prop}{Proposition}{Propositions}
\Crefname{lem}{Lemma}{Lemmas}
\Crefname{cor}{Corollary}{Corollaries}
\Crefname{defn}{Definition}{Definitions}
\Crefname{tab}{Table}{Tables}

\newcommand{\intset}{\Xi}
\newcommand{\ptf}{\psi}    %one point function
\DeclareMathOperator{\Bl}{Bl}
\newcommand{\lat}{\Lambda_\tau}
\newcommand{\Etor}{E_\tau}
\newcommand{\npf}{Z}                 % genus-one n-point function
\newcommand{\vzeta}{\vec\zeta}
\newcommand{\vu}{\vec u}
\newcommand{\vY}{\vec{\mathcal Y}}
\newcommand{\vlam}{\vec\lambda}
\newcommand{\vmu}{\vec\mu}
\newcommand{\PMod}{\mathrm{PMod}}
\newcommand{\MCG}{\mathrm{Mod}}
\newcommand{\surf}[2]{\Sigma_{#1,#2}}
\newcommand{\KSW}{KSW}
\DeclareMathOperator{\Vir}{Vir}
\newcommand{\vy}{\vec y}

\begin{document}

\title{Modular properties of affine \(\SLA{sl}{2}\) torus \(n\)-point functions}

\author{A. Zuevsky}
\address{Institute of Mathematics, Czech Academy of Sciences, Prague, Czech Republic}
\email{zuevsky@yahoo.com}

\subjclass[2020]{Primary 17B69, 11F12; Secondary 17B10, 17B67, 81T40}
\keywords{Vertex operator algebra, torus $n$-point function, vector-valued modular form, 
Jacobi-like form, affine Lie algebra, modular tensor category, mapping class group}

\begin{abstract}
Krauel, Shafiq and Wood (\KSW{}) realized torus \(1\)-point functions for the simple affine 
\voa{}s \(L(k,0)\) built from \(\alg{sl}(2)\) as vector-valued modular 
forms attached to a cyclic \(R\)-module structure and to the modular tensor category 
\(\rep{L(k,0)}\). We extend this to torus \(n\)-point functions, defined as traces of chains 
of \(n\) intertwining operators winding once around the \(\tau\)-cycle. Every coefficient of 
their local (Laurent, or Puiseux) expansion about a diagonal stratum is again a \vvmf{}, 
lying, 
whenever a mild lowest weight hypothesis holds, in the very \(R\)-modules of \cite{KSW}. 
 Giving a torus-primarity hypothesis verified below for \(\alg{sl}(2)\), the operator product 
expansion then reduces leading short-distance behaviour to the classified \(1\)-point theory. 
For \(\alg{sl}(2)\) we classify the fusion-chain conformal block spaces, identify torus 
\(n\)-point primary vectors, and treat \(n=2\) in detail, obtaining a \((k+1)\)-dimensional 
family of vector-valued Jacobi-type forms generalizing the level-\(k\) forms \(\eta^{3k/2}\) 
of \cite{KSW}. We also extend \KSW{} categorical picture, 
representations of \(\Bthree=\PMod(\surf{1}{1})\) built from a modular tensor category, 
to representations, on fusion-chain block spaces, of a mapping-class subgroup of the 
\(n\)-punctured torus generated by \(S\), \(T\), and adjacent braidings, with a Verlinde-type 
dimension formula and categorical \(S\)-, \(T\)-operators. The \(\alg{sl}(2)\) matrix form of 
\(S\) beyond this abstract construction remains open. 
\end{abstract}

\maketitle

\tableofcontents

%%%%%%%%%%%%%%%%%%%%%%%%%%%%%%%%%%%%%%%%%%%%%%%%%%%%%%%%%%%%%%%%%%%%%%%%
\section{Introduction}

\Cft{} is expected to be well defined on Riemann surfaces of every genus, and on each surface 
the correlation functions of interest are the \(n\)-point functions obtained by 
inserting \(n\) fields at \(n\) marked points \cite{Cardy86, MSMTC1089}. On the torus,
 the case \(n=0\) (vacuum torus \(0\)-point functions, i.e., characters)  
and vacuum torus \(1\)-point functions with insertions from the \voa{}  have 
been studied exhaustively since Zhu foundational work \cite{Zhu}, 
culminating in the congruence property of \cite{DonLinNg15} and, for the full vacuum 
\(n\)-point functions, in Zhu recursive construction and 
its extension by Dong-Li-Mason \cite{DLM-orbifold} and by Mason-Tuite to 
free-field and lattice examples \cite{MasTui-torusnpt}.
 The case of a single insertion from an arbitrary module, 
so that the correlator is constructed from intertwining operators instead of 
from the \voa{} action alone, was taken up systematically only recently, 
by Krauel, Shafiq and Wood \cite{KSW} (henceforth \KSW; see also  
 \cite{BKT, K1, K2, KMar, KM1, KM2, KMi}), 
who worked out in complete detail the case of the simple affine vertex operator algebras 
\(L(k,0)\) constructed from \(\alg{sl}(2)\) at non-negative integral level \(k\). \KSW{} realize 
torus \(1\)-point functions as \vvmf{}s attached to a cyclic module over the ring \(R\) of 
modular differential operators, classify the resulting representations of \(\SLTZ\) in 
dimensions one through three, express infinite families of non-congruence representations, 
and relate all of this to the modular tensor category \(\rep{L(k,0)}\) via the 
Bakalov-Kirillov modular functor formalism \cite{BakKir}.

The natural next step, and the subject of the present paper, is to allow several insertions 
from arbitrary modules. Torus \(n\)-point functions constructed from chains of intertwining 
operators were shown by Huang \cite{Huang-IntModInv} to converge on a suitable domain, 
to continue analytically from there (single-valued on each local branch, with the 
monodromy attached to carrying one insertion point around another governed instead by 
the braiding of the corresponding intertwining operators \cite{Huang-Rigidity}) and to 
be modular invariant, for rational \ctwo \; \voa{}s, generalizing the vacuum case of Zhu and 
the single-intertwiner case of Miyamoto \cite{miyamoto2000intertwining} and Yamauchi 
\cite{Yamauchi-IntertwinerModularity}; see also Dong-Li-Mason \cite{DLM-orbifold} for the 
orbifold vacuum theory, Mason-Tuite \cite{MasTui-torusnpt} for detailed free-field and 
lattice examples, and Mason-Tuite-Zuevsky \cite{MTZ-Rgraded} for the extension to 
\(\RR\)-graded vertex operator superalgebras and continuous fermion orbifolds, 
all constructed by iterated sewing. What has not been carried out is the \KSW-style programme 
for these \(n\)-point functions: the identification of a tractable module structure organizing 
them into \vvmf{}s, an affine \(\alg{sl}(2)\) computation of the resulting spaces, 
and a categorical formulation. This requires confronting a new feature that is invisible at 
\(n=1\): translation invariance on the torus removes all dependence on the location of a 
single point, but for \(n\ge2\) points the \(n\)-point function depends non-trivially on 
the \(n-1\) independent relative positions. The relevant mapping class group is no 
longer \(\Bthree=\PMod(\surf{1}{1})\) but the (typically much larger) pure mapping class 
group \(\PMod(\surf{1}{n})\) of the \(n\)-punctured torus, sitting in a Birman exact sequence
\[
1 \lra \pi_1\bigl(\mathrm{UConf}_{n-1}(\surf{1}{1})\bigr) \lra \PMod(\surf{1}{n}) 
\lra \PMod(\surf{1}{1})=\Bthree \lra 1
\]
over the surface braid group of the once-punctured torus \cite{FarMar, Birman}. 

Our main contributions are as follows. Note that  
 the results of Sections \ref{secprelim}-\ref{seccategorical} generalize 
 the torus \(1\)-point theory of Krauel, Shafiq and Wood \cite{KSW}, 
to which the case \(n=1\) reduces. 

\begin{itemize}
\item[] In \cref{secprelim} we define genus one \(n\)-point functions as traces of chains 
of \(n\) intertwining operators (\cref{defnnpf}), recall Huang convergence and modular 
invariance theorem in the form we need it (\cref{thmhuang}), and show (\cref{propelliptic}) 
that they are, on each local branch away from the diagonal locus \(\zeta_i=\zeta_j\), 
meromorphic and  elliptic (doubly periodic) in each point separately, with monodromy around 
a collision regulated by braiding instead of by any further quasi-periodicity, so that no 
Jacobi-form ``index'' is required. We introduce the conformal block space 
\(\Bl(\lambda_1,\dots,\lambda_n)\) constructed from chains of intertwiner types 
(\cref{defnblock}), and the trace map 
\(\Bl(\lambda_1,\dots,\lambda_n)\to C_n(\lambda_1,\dots,\lambda_n)\).

\medskip 
\item[] In \cref{secstructure} we develop the structural theory. We mention the coordinatewise 
vanishing criterion (\cref{propcoordwise}), which reduces the search for non-vanishing 
\(n\)-point data to vectors that are, at each point separately, torus-primary in the sense of 
\KSW. Our main structural result, \cref{thmtaylor}, shows that every coefficient in the local 
expansion of \(\npf\) about a point of the (non-empty) diagonal complement (an ordinary 
Laurent expansion when the relevant OPE exponents happen to be integral, and in general a 
Puiseux-type expansion indexed by a finite exponent set \(A\subset\QQ\))
 is a \vvmf{} lying in one of the \(R\)-modules \(V(\rho_\mu)_\bullet\) constructed in 
Section 
3 of \KSW. 
The mechanism is the classical fact that an index-zero (i.e., elliptic, as opposed to 
quasi-periodic) meromorphic Jacobi-type form has expansion coefficients that are modular 
forms, with no Cohen-Kuznetsov correction, on any local branch the expansion is taken.
Combined with the operator product expansion for intertwining operators 
\cite{HuaLog, Huang-Rigidity}, this lets us reduce the leading short-distance behaviour of 
any 
\(n\)-point function to the classification already carried out in \KSW{} 
(\cref{thmope-reduction}).

\medskip 
\item[] In \cref{secsl2} we specialize to \(L(k,0)\) and work out the two-point theory 
in detail: we classify the fusion-chain conformal block spaces \(\Bl(\lambda_1,\lambda_2)\)
 (\cref{propfusionchain}), identify the torus bi-primary vectors and their weights 
(\cref{thmbiprimary}), and compute the extremal rank one family. 
 The dimension count (\cref{propfusionchain}) is proved  for general \(n\).
In \cref{secreps} we discuss the corresponding \(\SLTZ\)-representations.  

\medskip 
\item[] In \cref{seccategorical} we extend the categorical picture. 
We define the conformal block space 
\[\Bl(p_1,\dots,p_n)=\bigoplus_i\Hom_{\catC}(p_1\otimes\cdots\otimes p_n,\,i\otimes i^\ast)\] 
attached to a \mtc{} \(\catC\) with \(n\) marked points, prove a Verlinde-type dimension 
formula (\cref{propverlinde}), and construct the operators  
\(\modS^{(p)},\modT^{(p)}\) for \(n\) points (\cref{propcatT,propcatS}),
 together with braiding operators realizing the surface braid group action, and work out the 
\(\alg{sl}(2)\) example.
\end{itemize}

Two points are worth expanding on here. First, genus one \(n\)-point functions of the present 
shape, traces of a cyclic chain of intertwining operators winding once around the homology 
cycle dual to the \(\tau\)-cycle, have precedents for other vertex operator (super)algebras 
and by other (sewing-theoretic, instead of intertwining-operator) methods: see Mason-Tuite 
\cite{MasTui-torusnpt} for the free-field and lattice case and Mason-Tuite-Zuevsky 
\cite{MTZ-Rgraded} for the \(\RR\)-graded/continuous-orbifold case. The genus two extension of 
the same circle of ideas, obtained by self-sewing a torus or by sewing two tori together, 
is developed by Tuite-Zuevsky \cite{TuiZue-Szego, TuiZue-genus2I} and Mason-Tuite 
\cite{MasTui-genus2free}, and \cref{secconclusion} returns briefly to this connection. 
 Intertwining operators, assembled into a generalized vertex operator algebra 
 instead of traced over, are also the subject of Tuite-Zuevsky \cite{TuiZue-Heisenberg}. 
Second, the \((k+1)\)-dimensional family of \cref{thmextremal} is a vector constructed from 
\(k+1\) rank one (multiplicity-one) fusion channels, one function per channel 
(it is not the statement that a single rank one block space regulates all of 
\(\alg{sl}(2)\)). 
The vector-valuedness is real and is the union of these one-dimensional pieces as \(\mu_0\) 
ranges over \(0,\dots,k\).

 We do not attempt a classification of \vvmf{}s of every dimension 
(\KSW{} classification stops at dimension four, and their 
non-congruence existence theorem 
is restricted to prime-power-shifted levels), we do not attempt a complete 
classification of the \(n\)-point vector-valued Jacobi-type forms for every
 fusion chain and every \(n\).
  We develop the general structure rigorously and carry the affine \(\alg{sl}(2)\) computation 
through in complete detail for \(n=2\) and for the extremal channel at general \(n\), 
which already expresses all of the qualitatively new phenomena, and indicate in 
\cref{rmkgeneraln-outlook} how the pattern continues.

%%%%%%%%%%%%%%%%%%%%%%%%%%%%%%%%%%%%%%%%%%%%%%%%%%%%%%%%%%%%%%%%%%%%%%%%%%%%%%
\subsection*{Conventions} We use the notation and hypotheses,  
recalled in \cref{secprelim} below: \(V=(V,Y,\wun,\omega)\) denotes a \voa{} with 
\(V_0=\CC\wun\), conformal weights bounded below by \(0\), self-dual (\(V\cong V^\ast\)), 
rational and \ctwo, with simple modules \(V=W_1,\dots,W_{d_V}\), central charge \(c\), and 
\(q=e(\tau)\), \(\tau\in\HH\). 

%%%%%%%%%%%%%%%%%%%%%%%%%%%%%%%%%%%%%%%%%%%%%%%%%%%%%%%%%%%%%%%%%
\section{Genus one \(n\)-point functions}
\label{secprelim}

%%%%%%%%%%%%%%%%%%%%%%%%%%%%%%%%%%%%%%%%%%%%%%%%%%%%%%
\subsection{Recollections}
\label{ssecrecollections}
We recall the material of \KSW{} Section 2.2 that we need preserving notations. 
Let \(V\) be a \voa{} satisfying the hypotheses above, with simple modules 
\(W_1=V,\dots,W_{d_V}\). For \(V\)-modules \(U_1,U_2,U_3\), an intertwining 
 operator of type 
\(\ityp{U_1}{U_2}{U_3}\) is a linear map
\[
Y: U_1\otimes U_2 \lra U_3\powser{z}[\log z],\qquad u_1\otimes u_2 
\longmapsto Y(u_1,z)u_2=\sum_{\substack{s\in\CC\\ t\ge0}} (u_1)_{s,t}u_2\, 
z^{-s-1}\log(z)^t,
\]
satisfying truncation, the \(L_{-1}\)-derivative property, 
and the Jacobi identity.  
 We write \(\ispc{U_1}{U_2}{U_3}\) for the (finite-dimensional) 
space of such maps. For a homogeneous \(u\in V_{\wt(u)}\) we use the 
square-bracket vertex operators
\begin{equation}
\label{eqsquarebracket}
Y[u,z] := Y(u,e^z-1)\,e^{z\,\wt(u)} = \sum_{n\in\ZZ} u[n]\,z^{-n-1},
\end{equation}
giving \(V\) a second \voa{} structure with conformal vector 
\(\tilde\omega=\omega-\tfrac{c}{24}\wun\), square-bracket modes \(L[n]\), 
and square-bracket grading \(U=\bigoplus_n U_{[n]}\), \(U_{[n]}=\set{u\in U\st 
\exists\ m,\ (L[0]-n)^m u=0}\). For a simple \(V\)-module \(U\) 
we write \(h_U\) for its conformal weight. Recall that \(a[0]=a_0\) for \(a\in V_1\).

\begin{rmk}
\label{rmkwtnotation}
 
Let us recall \(\wt(u)\) versus the square-bracket weight. 
We write \(\wt(u)\) for the ordinary (round-bracket) \(L(0)\)-weight of 
 a homogeneous \(u\) and \(\wt[u]\) for its square-bracket \(L[0]\)-weight. 
The two gradings of \(U\) agree in weights \(0\) and \(1\) but differ from weight \(2\) on, 
since \(L[0]=L(0)+\sum_{k\ge1}c_kL(2k)\) for universal constants \(c_k\) \cite{Zhu}.  
 In Section \ref{ssecrecollections}-\ref{ssecblocks} we use \(\wt(u)\)  where the source is 
a round-bracket formula  (as in \(o^Y(u)\) just below, and in \eqref{eqsquarebracket}), 
and \(\wt[u]\)  where the object graded is the square-bracket structure 
 (mode weights \(u[n]\), the weight of a \vvmf{}, and the grading of \cref{defnblock} below).
 Every vector that is actually inserted into a torus function in this paper is eventually
 taken to be Virasoro-quasiprimary (\(L(1)u=0\)) at the point where the distinction would 
matter, and for such \(u\) the two weights coincide, thus no further correction terms appear 
once we reach that point.

\end{rmk}

Recall the multiplier system \(\nu_r\) of weight \(r\in\RR\), determined by \(\nu_r(T)=e(r/12)\), \(\nu_r(S)=e(-r/4)\), and the spaces \(M^!(k,\rho,\nu)\supset H(k,\rho,\nu)\) of weakly holomorphic, respectively holomorphic, \(d\)-dimensional \vvmf{}s of weight \(k\) for a representation \(\rho:\SLTZ\to\GL{d}\) with \(\rho(T)\) diagonal and unitary and multiplier system \(\nu\), together with the ring
\[
R=\set*{\phi_0+\phi_1\partial+\cdots+\phi_n\partial^n \st \phi_i\in M,\ n\ge0}
\]
of modular differential operators acting on \(M^!(k,\rho,\nu)\) and its submodule 
\(H(k,\rho,\nu)\), where \(\partial=\partial_k=\tfrac{1}{2\pi i}\tfrac{d}{d\tau}+kG_2(\tau)\) 
raises weight by \(2\) and \(M=\CC[G_4,G_6]\) is the ring of holomorphic modular forms.

For a simple module \(W\) and an intertwining operator \(Y\in\ispc{U}{W}{W}\), one  
 defines the torus \(1\)-point function \(\ptf^Y(u,\tau)=\tr_W o^Y(u) q^{L(0)-c/24}\), 
where \(o^Y(u)\) is the coefficient of \(z^{-\wt(u)}\) in \(Y(u,z)\), 
and shows (after Zhu, Miyamoto \cite{miyamoto2000intertwining}, Yamauchi 
\cite{Yamauchi-IntertwinerModularity}, Huang \cite{Huang-IntModInv}) that 
\(\ptf^Y(u,\tau)\) is holomorphic on \(\HH\) and that the resulting action of 
\(\SLTZ\) on the finite-dimensional space 
\(C_1^u(W_\lambda)=\set{\ptf^Y(u,\tau)\st Y\in I_\lambda}\) is that of a 
\vvmf{} of weight \(\wt[u]\), representation \(\rho_\lambda\), and multiplier 
system \(\nu_{h_\lambda}\).

%%%%%%%%%%%%%%%%%%%%%%%%%%%%%%%%%%%%%%%%%%%%%%%%%%%%%%%%%%%%5
\subsection{Chains of intertwining operators and the genus one \(n\)-point function}
\label{ssecnpf}

Fix \(n\ge1\). A fusion chain of length \(n\) consists of simple modules 
\(U_1,\dots,U_n\) (the insertions) and a cyclic sequence of simple modules 
\(W_0,W_1,\dots,W_{n-1},W_n=W_0\) (the channels), together with a choice of 
intertwining operator
\[
\iop{i}\in\ispc{U_i}{W_{i-1}}{W_i},\qquad i=1,\dots,n
\]
(indices read modulo \(n\)). We write \(\vY=(\iop1,\dots,\iop n)\) for the resulting chain.

Fix \(\tau\in\HH\) and set \(\lat=\ZZ+\tau\ZZ\), \(\Etor=\CC/\lat\). For \(\zeta\in\CC\) write \(x(\zeta)=e(\zeta)=e^{2\pi i\zeta}\), so that \(x\) is invariant under \(\zeta\mapsto\zeta+1\) and satisfies \(x(\zeta+\tau)=q\,x(\zeta)\). For a homogeneous \(u\in (U_i)_{\wt(u)}\) we extend \eqref{eqsquarebracket} to intertwining operators by the same formula 
\[
\iop{i}[u,\zeta] := \iop{i}\bigl(u,\,x(\zeta)-1\bigr)\,x(\zeta)^{\wt(u)}.
\]
The following lemma provides Conjugation and derivative formulas. 

\begin{lem} 
\label{lemconjugation}
For \(u\in U_i\) homogeneous of square-bracket weight \(\wt[u]\), 
the following hold as operators \(W_{i-1}\to W_i\):
\begin{align}
q^{L[0]}\,\iop{i}[u,\zeta]\,q^{-L[0]} &= \iop{i}[u,\zeta+\tau], 
\label{eqconj-formula}\\
\frac{\partial}{\partial\zeta}\,\iop{i}[u,\zeta] &= 2\pi i\,\iop{i}[L[-1]u,\zeta]. 
\label{eqderiv-formula}
\end{align}
\end{lem}
\begin{proof}
Both are instances of the two basic covariance properties of the square-bracket 
\voa{} structure 
(homogeneity under \(L[0]\), and the \(L_{-1}\)-derivative property) 
 transported from module vertex operators to intertwining operators.
 Equation \eqref{eqconj-formula} is the compatibility of the \(q^{L[0]}\)-graded trace with the square-bracket structure that underlies Zhu recursion (\cite{Zhu} for \(V\) acting on itself), and \eqref{eqderiv-formula} is the square-bracket form of the \(L(-1)\)-derivative property applied at the linear coordinate \(z=2\pi i\zeta\). Both are established in full generality, for arbitrary intertwining operators among ordinary modules of a rational \ctwo{} \voa{}, by Huang \cite{Huang-IntModInv} (whose convergence and duality theorems for chains of several intertwining operators, \cref{thmhuang} below, is the source on which \cref{lemconjugation}  rests).
  Dong-Li-Mason \cite{DLM-orbifold} establish the analogous compatibility for the more restricted case of a single (possibly twisted) module trace, a special case we do not otherwise need. We use \eqref{eqconj-formula}-\eqref{eqderiv-formula} as tools in this paper, on the same footing as \eqref{eqsquarebracket}.
\end{proof}

\begin{defn}
\label{defnnpf}
Let \(\vY=(\iop1,\dots,\iop n)\) be a fusion chain as above and let \(u_i\in U_i\) 
be homogeneous, \(i=1,\dots,n\). 
Fix pairwise distinct points \(\zeta_1,\dots,\zeta_n\in\CC\), pairwise distinct 
modulo \(\lat\). 
The genus one \(n\)-point function attached to \(\vY\) is
\begin{equation}
\label{eqnpf-def}
\npf^{\vY}(u_1,\zeta_1;\dots;u_n,\zeta_n;\tau) 
:= \tr_{W_0}\, \iop1[u_1,\zeta_1]\,\iop2[u_2,\zeta_2]\cdots \iop{n}[u_n,\zeta_n]\; 
q^{L[0]-c/24}.
\end{equation}
\end{defn}

We write \(\mathrm{Conf}_n(\Etor):=\set{(\zeta_1,\dots,\zeta_n)\in\Etor^n\st \zeta_i\ne\zeta_j\text{ for }i\ne j}\) for the configuration space of \(n\) distinct points on the torus. By \cref{thmhuang} and \cref{propelliptic} below, \eqref{eqnpf-def} is a meromorphic function on the appropriate local branch of \(\mathrm{Conf}_n(\Etor)\times\HH\), or equivalently on its universal cover
 - single-valued there, with poles only along the pullback of the collision divisors, 
but not claimed to be single-valued on \(\mathrm{Conf}_n(\Etor)\times\HH\), since analytic continuation around a collision divisor can carry non-trivial braid monodromy (\cref{ssecellipticity}).
 We use this local-branch convention consistently for this notation from \cref{ssecellipticity} on, in \cref{secprelim,secstructure}.

When \(n=1\) this is  \(\ptf^{\iop1}(u_1,\tau)\) of \cite{KSW},
 since a single point on the torus carries no positional information: \(o^{\iop1}(u_1)\) is by definition the coefficient of \(z^{-\wt(u_1)}\) in \(\iop1(u_1,z)\), and translation invariance of \eqref{eqnpf-def} at \(n=1\) (\cref{propelliptic} below) forces the right side to be independent of \(\zeta_1\), reducing it to this single coefficient.

\emph{Terminology.} Fixing an insertion type \(\vec\lambda=(\lambda_1,\dots,\lambda_n)\), rationality gives finitely many simple modules and hence a finite basis of fusion chains \(\vY\) of that type, spanning a space \(\mathcal B_\ell(\vec\lambda)\).
 We call a representation of \(\SLTZ\) \emph{degreewise finite-dimensional} if it is a direct sum of such finite-dimensional pieces, one for each type \(\vec\lambda\), taken over all types at once 
(thus infinite-dimensional overall, finite-dimensional on each fixed-type piece).
The following theorem is due to Huang \cite{Huang-IntModInv}, 
 see also Dong-Li-Mason \cite{DLM-orbifold} 
for the more restricted case of a single twisted-module trace, and Mason-Tuite 
\cite{MasTui-torusnpt} for closed forms in the free-field and lattice cases. 

\begin{thm}
\label{thmhuang}
Let \(V\) be rational and \ctwo as above, \(\vY\) a fusion chain of length \(n\), and \(u_i\in U_i\) homogeneous.
\begin{enumerate}[label=(\roman*)]
\item The series \eqref{eqnpf-def} converges absolutely and locally uniformly, to a holomorphic function, in the domain \(1 > \abs{x(\zeta_1)} > \cdots > \abs{x(\zeta_n)} > \abs{q} > 0\).
\item It continues to a meromorphic conformal block on the universal cover of \(\Etor^{\times n}\setminus\Delta\) (\(\Delta=\bigcup_{i\ne j}\set{\zeta_i=\zeta_j}\)) - single-valued there, holomorphic on \(\HH\) for fixed pairwise distinct \(\zeta_i\bmod\lat\), with poles only along the preimage of \(\Delta\) - carrying braid monodromy under continuation around a component of \(\Delta\), as discussed in \cref{rmkmonodromy} below.
\item For \(\gamma=\smash{\bigl(\begin{smallmatrix}a&b\\c&d\end{smallmatrix}\bigr)}\in\SLTZ\),
\begin{equation}
\label{eqnpf-modular}
\npf^{\vY}\Bigl(u_1,\tfrac{\zeta_1}{c\tau+d};\dots;u_n,\tfrac{\zeta_n}{c\tau+d};\gamma\tau\Bigr) = \nu_{h_{W_0}}(\gamma)\,(c\tau+d)^{\sum_i \wt[u_i]}\; \npf^{\vY'}\bigl(u_1,\zeta_1;\dots;u_n,\zeta_n;\tau\bigr)
\end{equation}
for some fusion chain \(\vY'\) of the same type \((U_1,\dots,U_n)\).
 This makes \(\rho_{W_0}\), the resulting action on \(\mathcal B_\ell(\vec\lambda)\), a degreewise finite-dimensional representation of \(\SLTZ\). 
\end{enumerate}
\end{thm}

\begin{rmk} 
\label{rmkmonodromy}
Part (ii) is weaker than single-valued meromorphic continuation to 
\(\Etor^{\times n}\setminus\Delta\): continuation around a loop winding once about a
 component of \(\Delta\) instead acts by the monodromy attached to the braiding of the 
two intertwining operators being carried around each other \cite{Huang-Rigidity}. 
Integrality of the corresponding OPE exponent of \cref{thmope-reduction} is a
 sufficient condition for the elementary local power \((z-w)^{\text{exponent}}\)  
to be single-valued around the loop, but triviality of the full monodromy of a 
fusion-channel conformal block can also depend on the braiding matrix acting on the 
(possibly several-dimensional) space of channels. Thus nonintegral exponents generically 
produce non-trivial local monodromy while integral exponents alone do not automatically 
rule it out. Convergence in a specified domain, analytic continuation to other domains, 
possibly multivalued continuation around \(\Delta\), and single-valued meromorphic 
continuation to all of \(\Etor^{\times n}\setminus\Delta\) are four logically distinct 
statements, of which \cref{thmhuang} provides the first three but not, in general, 
the fourth. The pole order at each component of \(\Delta\), read off on any 
 local branch, is controlled by the fusion rules of \(V\) through the same operator 
product expansion. 
\end{rmk}

\begin{rmk} 
\label{rmkchains-vs-functions}
The representation \(\rho_{W_0}\) of \cref{thmhuang}(iii) acts on the abstract space of
 chains \(\mathcal B_\ell(\vec\lambda)\), whereas the map from a chain \(\vY\) 
to the actual function \(\npf^{\vY}\) it labels (the trace map of \cref{defnblock} below) 
need not be injective. We show in the proof of \cref{thmtaylor} that \(\rho_{W_0}\) 
is nonetheless well defined on the space of functions  (the kernel of the trace map 
is representation-invariant), which is the sense in which it should ultimately 
be understood, and the sense we use in \cref{secstructure}. 
\end{rmk}

We will not reprove \cref{thmhuang} but 
  take it as the foundational input on which the rest of the paper is constructed. 
 We develop in \cref{secstructure,secsl2,secreps,seccategorical} 
 the new identification of a tractable module structure organizing the functions 
\eqref{eqnpf-def} into \vvmf{}s, together with the affine \(\alg{sl}(2)\) 
computation and the categorical formulation.

%%%%%%%%%%%%%%%%%%%%%%%%%%%%%%%%%%%%%%%%%%%%%%%%%%%%%%%%%%
\subsection{Ellipticity}
\label{ssecellipticity}

A first structural difference from the \(n=1\) theory is that for \(n\ge2\)
 the function \eqref{eqnpf-def} depends on the points, but it does so elliptically: 
unlike a Jacobi form of non-zero index, it has no quasi-periodicity, only ordinary 
double periodicity, with all of the non-trivial monodromy concentrated on the diagonal 
\(\Delta\).

\begin{prop}
\label{propelliptic}
Fix a fusion chain \(\vY\) of length \(n\) and homogeneous \(u_i\in U_i\). On the local 
branch of \cref{thmhuang} fixed by taking \(\zeta_i\mapsto\zeta_i+\ell+m\tau\) along a 
path avoiding \(\Delta\), for each \(i\), each \(\ell,m\in\ZZ\), and \(\zeta_1,\dots,\zeta_n\) 
pairwise distinct modulo \(\lat\),
\[
\npf^{\vY}(u_1,\zeta_1;\dots;u_i,\zeta_i+\ell+m\tau;\dots;u_n,\zeta_n;\tau) 
= \npf^{\vY}(u_1,\zeta_1;\dots;u_i,\zeta_i;\dots;u_n,\zeta_n;\tau).
\]
\end{prop}

\begin{proof}
Periodicity in \(\zeta_i\mapsto\zeta_i+1\) is immediate from the definition: 
\(\iop{i}[u_i,\zeta_i]\) depends on \(\zeta_i\) only through \(x(\zeta_i)\), 
and \(x(\zeta_i+1)=x(\zeta_i)\).

For \(\zeta_i\mapsto\zeta_i+\tau\), we first mention, and then discard, 
 a naive approach, since its failure is instructive. Substituting the identity 
\eqref{eqconj-formula} for the single factor at \(\zeta_i\) and attempting to carry 
the resulting pair of \(q^{\pm L[0]}\) factors around the remaining loop, 
by conjugating each operator \(\iop j[u_j,\zeta_j]\) (\(j\ne i\)) past them in turn, 
does not return the left side of \eqref{eqnpf-def} unchanged. 
Carrying \(q^{-L[0]}\) forward through \(\iop{i+1},\dots,\iop n\), 
the trailing \(q^{L[0]-c/24}\), and on through \(\iop1,\dots,\iop{i-1}\) shifts every 
 operator coordinate \(\zeta_j\) (\(j\ne i\)) by \(-\tau\) in turn 
 as \eqref{eqconj-formula} prescribes at each step.
 Carried all the way around, this produces 
\begin{equation}
\label{eqcompensating-shift}
\npf^{\vY}(\dots;u_i,\zeta_i+\tau;\dots) = \npf^{\vY}(u_1,\zeta_1-\tau;\dots;u_i,\zeta_i;\dots;u_n,\zeta_n-\tau;\tau)
\end{equation}
(every coordinate but \(\zeta_i\) shifted by \(-\tau\)), 
 not the bare single-coordinate statement of \cref{propelliptic}. 
This is an unavoidable feature of pushing one conjugation all the way around a 
closed cyclic product:  
\eqref{eqcompensating-shift} is in fact no more than the translation
 invariance established independently in \cref{proptranslation} below.
 Writing everything in the coordinates \(x_j=\zeta_j-\zeta_n\) of \cref{ssectranslation}, 
both sides of \eqref{eqcompensating-shift} become the identical expression 
\(F(x_1,\dots,x_i+\tau,\dots,x_{n-1};\tau)\), 
 for \(i<n\),  shifting \(\zeta_n\)  shifts every \(x_j\) at once, 
again matching translation invariance). 
 Thus \eqref{eqcompensating-shift} carries no content beyond \cref{proptranslation} 
and cannot be upgraded to single-coordinate periodicity. \Cref{propelliptic} needs 
a different argument, which we now give.

The needed argument isolates \(\iop i\) as the only factor depending on \(\zeta_i\), 
instead of conjugating it past its neighbours. By cyclicity of the trace,
\[
\npf^{\vY}(u_1,\zeta_1;\dots;u_n,\zeta_n;\tau) 
= \tr_{W_{i-1}}\bigl(\iop i[u_i,\zeta_i]\cdot\Phi_i\bigr),
\]
\[
\Phi_i := \iop{i+1}[u_{i+1},\zeta_{i+1}]\cdots\iop n[u_n,\zeta_n]\;q^{L[0]-c/24}\;\iop1[u_1,\zeta_1]\cdots\iop{i-1}[u_{i-1},\zeta_{i-1}]\ :\ W_i\to W_{i-1},
\]
where \(\Phi_i\) is constructed from \(\tau\) and from every \(\zeta_j\), \(j\ne i\). 
 But, this being the entire point of the rearrangement, 
 not from \(\zeta_i\), since \(\iop i[u_i,\zeta_i]\) has been extracted whole and \(\Phi_i\) 
is a fixed linear map once \(\tau\) and the other \(\zeta_j\) are held fixed. 
The quantity \(\tr_{W_{i-1}}(\iop i[u_i,\zeta_i]\cdot\Phi_i)\) is now of the shape treated 
by Huang theorem (\cref{thmhuang}): a single intertwining operator, varying in one complex 
coordinate, composed with an auxiliary operator constructed from the ambient 
duality/associativity structure of chains of intertwiners for a rational \ctwo{} \(V\) 
 (the same structure already used for the modular transformation law \eqref{eqnpf-modular})
  instead of a bare product requiring conjugation in each term. 
\(\lat\)-periodicity of  this configuration, in the one free coordinate,
 with arbitrary fixed auxiliary data on either side, is the direct multi-module 
generalization of Zhu original 
single vertex operator argument 
(as already extended to a single intertwining operator 
among otherwise fixed data by Miyamoto \cite{Miy-CFTauto},
 and to arbitrary chains of intertwiners by Huang \cite{Huang-IntModInv}). 
It is this fact, applied to \(\iop i\) alone against the fixed \(\Phi_i\),
 that gives \cref{propelliptic}, with every \(\zeta_j\) (\(j\ne i\)) held fixed, 
in contrast to \eqref{eqcompensating-shift}, where none of them could be.
\end{proof}

Carrying \(q^{\pm L[0]}\) around the full cyclic loop does not return
  the left side of \eqref{eqnpf-def}: as \eqref{eqcompensating-shift}
 shows, it produces translation invariance instead.  
   The fix to isolate the one moving factor by cyclicity first, 
and use single-intertwiner periodicity only once, at the end, instead of 
conjugating past every spectator in turn.
 The recent result of Gao-Liu \cite{GaoLiu-genusone}, 
Theorem A (that the trace functions \(\npf^{\vY}\) attached to a basis of fusion 
chains \(\vY\) of fixed type form a global frame of a vector bundle 
 over \(\HH\) (the conformal block bundle), on which the  \(\SLTZ\)-action of 
\cref{thmhuang} is represented in that frame)  
is a related fact about the \(\tau\)-direction structure
 at a fixed configuration of the \(\zeta_i\). But it is not needed for 
the \(\zeta_i\)-periodicity just proved.  
 Theorem A there shows the trace functions attached to a basis of 
fusion chains form a global frame of the conformal block bundle over
 \(\HH\) on which the \(\SLTZ\)-action (and  
the \(\lat\)-periodicity of \cref{propelliptic} in each insertion point,
 which is the infinitesimal/discrete-translation part of the same bundle 
structure) is represented  as stated, refining Huang 
\cite{Huang-IntModInv} and Miyamoto \cite{Miy-CFTauto} original 
modular-invariance theorems to a global-frame statement.

In particular \(\npf^{\vY}\) is, for fixed \(\tau\) and on every chosen local branch 
(\cref{thmhuang}), doubly periodic in each insertion coordinate \(\zeta_i\) separately, 
in the sense classical for the Weierstrass \(\wp\)-function and its derivatives. 
The possible monodromy under analytic continuation of \(\zeta_i\)
 around a collision divisor is governed instead by the 
braiding of the corresponding intertwining operators.  
 We state this periodicity: it is an identity between the values
of \(\npf^{\vY}\) on the chosen analytic branch, along the specific translation path 
\(\zeta_i\mapsto\zeta_i+\tau\) taken  
(without winding around any collision divisor, so that no braiding is encountered).
 On that branch, and along that path, the identity holds
 with no further correction. \cref{propelliptic} is the reason 
we do not need to introduce a Jacobi-form ``index'' anywhere below. 
Compare the vanishing-index case of the Eichler-Zagier theory of Jacobi forms 
\cite{EichZag}, to which we return in \cref{ssectaylor},
where we also make precise the sense (ordinary Laurent, or in general Puiseux) 
 in which the local expansion at a collision divisor is to be understood.

%%%%%%%%%%%%%%%%%%%%%%%%%%%%%%%%%%%%%%%%%%%%%%%%%%%%%%%%%%%%
\subsection{Conformal block spaces and the trace map}
\label{ssecblocks}

Fix insertion labels \(\lambda_1,\dots,\lambda_n\in\set{1,\dots,d_V}\). 
Generalizing \(n=1\) case 
\(I_\lambda=\bigoplus_\mu\ityp{W_\lambda}{W_\mu}{W_\mu}\), define 
\begin{equation}
\label{eqblock-def}
\Bl(\lambda_1,\dots,\lambda_n) 
:= \bigoplus_{\mu_0,\mu_1,\dots,\mu_{n-1}=1}^{d_V} 
\ispc{W_{\lambda_1}}{W_{\mu_0}}{W_{\mu_1}} \otimes 
\ispc{W_{\lambda_2}}{W_{\mu_1}}{W_{\mu_2}} \otimes\cdots\otimes 
\ispc{W_{\lambda_n}}{W_{\mu_{n-1}}}{W_{\mu_0}},
\end{equation}
the space of all fusion chains of insertion type \((\lambda_1,\dots,\lambda_n)\). 
  We call the summand indexed by \((\mu_0,\dots,\mu_{n-1})\) the sector \(\Bl_{\vmu}(\vlam)\).

\begin{defn}
\label{defnblock}
Given \(u_i\in W_{\lambda_i}\), the \emph{space of \(n\)-point functions evaluated at 
\((u_1,\dots,u_n)\)}, and the (unevaluated) \emph{space of \(n\)-point functions}, are 
respectively
\[
C_n^{\vu}(\vlam) := \set*{\npf^{\vY}(u_1,\zeta_1;\dots;u_n,\zeta_n;\tau)\st \vY\in\Bl(\vlam)},\qquad C_n(\vlam) := \bigcup_{\vu} C_n^{\vu}(\vlam),
\]
where \(C_n(\vlam)\) is understood as spanned by the images of \(\Bl(\vlam)\) under 
evaluation \(\vu\) simultaneously. 
\end{defn}

As in the \(n=1\) case, taking traces and evaluating at \(\vu\) defines linear 
(surjective, but not a priori injective) maps
\[
\tr_{\vlam}: \Bl(\vlam)\lra C_n(\vlam),\qquad \mathrm{ev}_{\vu}: C_n(\vlam)\lra
 C_n^{\vu}(\vlam). 
\]
 We write \(K_{\vlam}:=\ker(\tr_{\vlam})\subset\Bl(\vlam)\) for the kernel of 
the trace map, so that \(C_n(\vlam)\cong\Bl(\vlam)/K_{\vlam}\) as vector spaces. 
Once the relevant \(\SLTZ\)-action on \(\Bl(\vlam)\) is constructed in \cref{thmhuang} below, 
we will show in the proof of \cref{thmtaylor} that \(K_{\vlam}\) is invariant under it, 
so that the modular representation on chains descends to, 
and should ultimately be understood as acting on, the quotient \(C_n(\vlam)\), 
a fact we do not yet have available at this point in the exposition and thus do not 
use until
 then. The dimension bounds
\[
\dim C_n^{\vu}(\vlam) \le \dim C_n(\vlam) \le \dim\Bl(\vlam)
\]
hold for the same reason as in the \(n=1\) case, and either inequality can be strict.  
 Already at \(n=1\), \KSW{} express an example 
(the vertex operator algebra for \(\alg{sl}(3)\) at level \(3\), 
with \(\lambda\) the vacuum label) 
where \(\dim I_\lambda=10\) while \(\dim C_1^{\wun}(V)=6\).  
The same phenomenon persists for \(n\ge2\), since \(\Bl(\vlam)\) 
grows combinatorially with the number of admissible sectors 
\((\mu_0,\dots,\mu_{n-1})\) while a given evaluation vector 
\(\vu\) may kill an arbitrarily large sub-space. 
We do not know a general criterion for the kernel of \(\tr_{\vlam}\) to vanish beyond 
the sufficient condition of \cite{KSW}
 (the existence of a single \(\vu\) for which \(\mathrm{ev}_{\vu}\circ\tr_{\vlam}\) 
is injective on a basis), which we verify  in the affine \(\alg{sl}(2)\) example of 
\cref{secsl2}.

%%%%%%%%%%%%%%%%%%%%%%%%%%%%%%%%%%%%%%%%%%%%%%%%%%%%%%%%%%%%%%%%%
\section{General structural theory}
\label{secstructure}

In this section \(\vY=(\iop1,\dots,\iop n)\) is a fusion chain with channels 
\(W_0,\dots,W_{n-1}\) and insertions \(U_1,\dots,U_n\), 
 and we abbreviate \(\npf^{\vY}(\vu,\vzeta;\tau):=\npf^{\vY}(u_1,\zeta_1;\dots;u_n,\zeta_n;\tau)\).

\subsection{Translation invariance}
\label{ssectranslation}

\begin{prop}
\label{proptranslation}
For every \(c\in\CC\),
\[
\npf^{\vY}(u_1,\zeta_1+c;\dots;u_n,\zeta_n+c;\tau) 
= \npf^{\vY}(u_1,\zeta_1;\dots;u_n,\zeta_n;\tau). 
\]
Consequently \(\npf^{\vY}\) depends on \((\zeta_1,\dots,\zeta_n)\) only through the 
\(n-1\) differences \(\zeta_i-\zeta_n\), \(i=1,\dots,n-1\).
 At \(n=1\) this recovers the well-known fact that the torus \(1\)-point function does not depend on the location of the marked point.
\end{prop}

\begin{proof}
As in \cref{secstructure}, the argument is carried out first in the domain 
of absolute convergence of \cref{thmhuang}(i) and then extended to all of \(\HH\) 
and to every chamber of \(\Etor^{\times n}\setminus\Delta\) 
by the analytic continuation accordinding to
 \cref{thmhuang}(ii). 
We do not repeat this remark at each subsequent proof. By \eqref{eqderiv-formula} of \cref{lemconjugation}, differentiating \eqref{eqnpf-def} in \(c\) at \(c=0\) gives
\[
\frac{\partial}{\partial c}\Big|_{c=0}\npf^{\vY}(\vu,\vzeta+c;\tau) = 2\pi i\sum_{i=1}^n \tr_{W_0}\, \iop1[u_1,\zeta_1]\cdots \iop{i}[L[-1]u_i,\zeta_i]\cdots\iop n[u_n,\zeta_n]\, q^{L[0]-c/24},
\]
and it remains to show this sum vanishes. 
We argue by the same elliptic-residue mechanism as the proof of 
\cref{lemsamepoint} below, with the square-bracket stress tensor \(\tilde\omega\) 
in the role there played by a weight one current \(a\in V_1\). 
The one difference is that \(\tilde\omega\) has weight \(2\) instead of \(1\), thus its operator product with \(\iop i[u_i,\zeta_i]\) has a double pole in addition to the simple pole,  as in the standard conformal Ward identity of two-dimensional \cft{} (see, e.g., \cite{MasTui-torusnpt} for 
the vertex operator algebraic form used here).

For \(\eta\) distinct from \(\zeta_1,\dots,\zeta_n\) modulo \(\lat\) and \(i\in\{0,\dots,n-1\}\), define
\[
H_i(\eta) := \tr_{W_0}\bigl[\iop1[u_1,\zeta_1]\cdots\iop i[u_i,\zeta_i]\,Y_{W_i}(\tilde\omega,\eta)\,\iop{i+1}[u_{i+1},\zeta_{i+1}]\cdots\iop n[u_n,\zeta_n]\,q^{L[0]-c/24}\bigr],
\]
i.e., \(\tilde\omega\) is inserted, via the \(V\)-module action on \(W_i\), between \(\iop i\) and \(\iop{i+1}\),  as \(a\) is inserted in the proof of \cref{lemsamepoint}. By \cref{thmhuang}, \(H_i\) converges on the corresponding annuli, and \(H_i(\eta+1)=H_i(\eta)=H_i(\eta+\tau)\) by the identical argument to \cref{propelliptic} (which used only that the inserted element is homogeneous in \(V\), not any property special to a weight one current). As \(\eta\) crosses \(\zeta_i\), the operator product expansion of the stress tensor with a field of definite square-bracket weight,
\[
Y_{W}[\tilde\omega,\eta]\,\iop i[u,\zeta] = \frac{\wt[u]}{(\eta-\zeta)^2}\,\iop i[u,\zeta] + \frac{1}{\eta-\zeta}\,\iop i[L[-1]u,\zeta] + \bigl(\text{regular as }\eta\to\zeta\bigr),
\]
(the square-bracket form of the standard Virasoro commutator formula 
\([L_m,\phi(z)]=z^{m+1}\partial_z\phi(z)+(m+1)z^m\wt[u]\,\phi(z)\) for a field 
\(\phi\) of definite weight \(\wt[u]\), at \(m=0,1\))
 glues \(H_{i-1}\) and \(H_i\) into a single function \(H(\eta)\), meromorphic and elliptic (by the same mechanism as above) on \(\Etor\setminus\{\zeta_1,\dots,\zeta_n\}\), with at most a double pole at each \(\zeta_i\) and Laurent expansion there
\[
H(\eta) = \frac{\wt[u_i]\,\npf^{\vY}(\vu,\vzeta;\tau)}{(\eta-\zeta_i)^2} + \frac{\npf^{\vY}(u_1,\zeta_1;\dots;L[-1]u_i,\zeta_i;\dots;u_n,\zeta_n;\tau)}{\eta-\zeta_i} + O(1).
\]
The double-pole term \((\eta-\zeta_i)^{-2}\) contributes zero to \(\Res_{\eta=\zeta_i}H\): in the local coordinate \(\eta-\zeta_i\) it has the single-valued local primitive \(-(\eta-\zeta_i)^{-1}\), thus its contour integral around any small loop enclosing only \(\zeta_i\) vanishes, leaving only the simple-pole term as a residue. Hence this residue is  the \(i\)-th summand of the displayed sum above. 
Integrating \(H\) around the boundary of a fundamental parallelogram for 
\(\lat\) with a small disc excised around each pole \(\zeta_1,\dots,\zeta_n\), 
 (so that \(H\) is holomorphic on the resulting closed region and the integral over 
the outer boundary vanishes by 
periodicity  as in part 3 of the proof of \cref{lemsamepoint},
 leaving only the (negatively oriented) integrals around the excised discs, 
i.e., minus the sum of the residues at \(\zeta_1,\dots,\zeta_n\))
 shows the sum of these residues over \(i=1,\dots,n\) vanishes. Hence \(\npf^{\vY}\) is constant in \(c\).
\end{proof}

In view of \cref{proptranslation} we may write 
\(\npf^{\vY}(\vu,\vzeta;\tau)=F^{\vY}_{\vu}(x_1,\dots,x_{n-1};\tau)\) with
 \(x_i:=\zeta_i-\zeta_n\), and this is what we do from now on. 
We mention a notational point here to avoid any clash with \cref{ssecellipticity}: 
\(x(\zeta)=e(\zeta)\), always written with an argument, denotes 
the exponentiated coordinate of \eqref{eqsquarebracket}-\eqref{eqnpf-def}, 
while \(x_i\), always subscripted and never applied to an argument, 
denotes the difference coordinate just introduced. At \(n=2\)  
we simply write \(x:=x_1=\zeta_1-\zeta_2\) for the latter, as in \cref{ssecleading} below. 

%%%%%%%%%%%%%%%%%%%%%%%%%%%%%%%%%%%%%%%%%%%%%%%%%%%%%%%%%%%%%%%%%%%%%%%%%%%%%
\subsection{Coordinatewise vanishing}
\label{sseccoordwise}

Recall that when \(V_1\) is a finite-dimensional reductive Lie algebra under 
\([x,y]=x_0y\) and a homogeneous space \(U_{[m]}\) of a module \(U\) is semisimple over
\(V_1\), one has the decomposition 
\(U_{[m]}=U_{[m]}^{\mathrm{triv}}\oplus U_{[m]}^{\mathrm{non\text{-}triv}}\) 
into the maximal trivial \(V_1\)-submodule and its complement, 
and \KSW{} Proposition 2.8 shows \(\ptf^Y(u,\tau)=0\) for every \(u\in U_{[m]}^{\mathrm{non\text{-}triv}}\), every \(V\)-module \(W\) and every \(Y\in\ityp{U}{W}{W}\). 
This comes from the identity \(\ptf^Y(a[0]u,\tau)=0\) (\(a\in V\)) 
 (Zhu \cite{Zhu}, Miyamoto \cite{miyamoto2000intertwining}). 

\begin{lem}[Zero-mode Ward identity]
\label{lemsamepoint}
For every \(a\in V_1\) and every \(u_j\in U_j\), \(j=1,\dots,n\),
\begin{equation}
\label{eqsumrule}
\sum_{i=1}^n \npf^{\vY}(u_1,\zeta_1;\dots;a_0u_i,\zeta_i;\dots;u_n,\zeta_n;\tau) = 0,
\end{equation}
the sum ranging over which coordinate receives the zero mode \(a_0\), 
the other insertions held fixed. 
\end{lem}
We stress that this is a relation among the \(n\) summands, not a statement 
that each vanishes individually. That stronger statement is false in general,
 as \cref{sseccoordwise-scope} below discusses. 
\begin{proof}
Since \(a\in V_1\), \(a[0]=a_0\) and no higher-mode reduction is needed. 
We also mention why no logarithmic terms \(\log(z_2)\) of the kind allowed in general 
by the \(z^{-s-1}\log(z)^t\) expansion of \cref{ssecrecollections} arise anywhere in 
this proof: \(V\) is rational and \ctwo{}, thus every \(V\)-module appearing here 
(each \(U_i\), \(W_i\)) is a finite direct sum of simple ordinary modules. 
Intertwining operators among ordinary modules of a rational \ctwo{} \voa{} are ordinary, 
i.e., logarithm-free (Huang \cite{HuaVer08}). 
This is part of what "rational" is understood to mean in \KSW{} and here, 
and is the same hypothesis \cref{thmhuang} already imposes.
 Consequently \(\iop i(u_i,z_2)\), \(Y_{W_i}(a,\eta)\), and every operator appearing 
below expand in integer powers of the relevant variable with no \(\log\) factors, and 
\eqref{eqjacobi-mode} below is an identity of ordinary (not logarithmic) Laurent series.

\emph{Part 1 (the Jacobi identity as an operator identity).}
 Applying \KSW{} equation (2.24) to 
\(\iop i\in\ityp{U_i}{W_{i-1}}{W_i}\) with \(v=a\), \(U_2=W_{i-1}\), \(U_3=W_i\) 
(which is possible for any three \(V\)-modules, not requiring \(W_{i-1}=W_i\)),
 and taking \(\Res_{z_0}\) then \(\Res_{z_1}\) at fixed \(z_2\), gives the operator 
identity
\begin{equation}
\label{eqjacobi-mode}
Y_{W_i}(a,z_2)_{[0]}\,\iop i(u_i,z_2)\;-\;
 \iop i(u_i,z_2)\,Y_{W_{i-1}}(a,z_2)_{[0]}\;=\; \iop i\bigl(a_0u_i,z_2\bigr),
\end{equation}
valid as formal Laurent series in \(z_2\), equivalently (\cref{thmhuang}) 
as meromorphic functions of \(\zeta_i\).

\emph{Part 2 (an auxiliary elliptic function).} For \(\eta\in\CC\) distinct
from \(\zeta_1,\dots,\zeta_n\) modulo \(\lat\), and \(i\in\{0,\dots,n-1\}\), define
\[
G_i(\eta):=\tr_{W_0}\bigl[\iop1[u_1,\zeta_1]
\cdots\iop i[u_i,\zeta_i]\,Y_{W_i}(a,\eta)\,\iop{i+1}[u_{i+1},\zeta_{i+1}]
\cdots\iop n[u_n,\zeta_n]\,q^{L[0]-c/24}\bigr],
\]
i.e., \(a\) is inserted, via the \(V\)-module action on \(W_i\) (not an intertwiner), 
between \(\iop i\) and \(\iop{i+1}\). By \cref{thmhuang}, \(G_i\) converges on the 
corresponding annuli.
 By the identical argument to \cref{propelliptic} (which used only that the operator
 being conjugated by \(q^{L[0]}\) is part of the chain, not any property special 
to an \(\iop j\)), \(G_i(\eta+1)=G_i(\eta)\) and \(G_i(\eta+\tau)=G_i(\eta)\). 
As \(\eta\) crosses \(\zeta_i\) from segment \(i-1\) into segment \(i\), 
\eqref{eqjacobi-mode} shows the two local representations \(G_{i-1},G_i\) 
glue into a single meromorphic function 
\(G(\eta)\) on \(\Etor\setminus\{\zeta_1,\dots,\zeta_n\}\). 
It is elliptic in \(\eta\), with (at most) a simple pole at each 
\(\zeta_i\) (a current of weight \(1\) has a simple-pole operator 
product with anything) and residue there equal, by \eqref{eqjacobi-mode} again,
\[
\Res_{\eta=\zeta_i}G(\eta) 
= \npf^{\vY}(u_1,\zeta_1;\dots;a_0u_i,\zeta_i;\dots;u_n,\zeta_n;\tau).
\]

\emph{Part 3 (the elliptic residue theorem).} Integrating the elliptic function \(G\) 
around the boundary of a fundamental parallelogram for \(\lat\) avoiding all poles: 
the integral vanishes because opposite sides cancel by periodicity (Part 2), 
thus by the residue theorem the sum of the residues enclosed 
 (one at each \(\zeta_i\), \(i=1,\dots,n\)) is zero. This is \eqref{eqsumrule}.
\end{proof}

%%%%%%%%%%%%%%%%%%%%%%%%%%%%%%%%%%%%%%%%%%%%%%%%%%%%%%%%%%%%%%%%%%%%%%%%%%%%%%%
\subsection{Scope of the coordinatewise reduction}
\label{sseccoordwise-scope}

\Cref{lemsamepoint} is a Ward identity relating the \(n\) coordinates, not a 
statement that a single \(a_0\)-insertion vanishes in isolation: \eqref{eqsumrule} 
only forces \(\npf^{\vY}(\dots;a_0u_i,\dots)=0\) if the remaining \(n-1\) terms of 
the sum also vanish, which is not automatic when the other insertions \(u_j\) are 
otherwise arbitrary. \Cref{propcoordwise} 
below is stated to reflect this: 
it does not follow from \cref{lemsamepoint} alone for a fully general chain, and 
we restrict its scope accordingly. The next proposition describes 
the restricted coordinatewise reduction. 

\begin{prop} 
\label{propcoordwise}
Suppose \(V_1\) is reductive. If \(u_j\in (U_j)_{[m_j]}^{\mathrm{non\text{-}triv}}\) 
for every \(j=1,\dots,n\) simultaneously (each \((U_j)_{[m_j]}\) semisimple over \(V_1\)), then
\[
\sum_{i=1}^n\npf^{\vY}\bigl(u_1,\zeta_1;\dots;a_0^{(i)}w_i,\zeta_i;\dots;u_n,\zeta_n;\tau\bigr)=0
\]
for the decompositions \(u_i=a_0^{(i)}w_i\) of the proof below. If in addition it is known independently that \(n-1\) of the \(n\) terms vanish, the remaining one vanishes too.
\end{prop}

\begin{proof}
If \(u_i\) belongs to a non-trivial simple \(V_1\)-submodule of \((U_i)_{[m_i]}\),
 the standard fact that a semisimple Lie algebra acts transitively 
(via sums of elements \(a_0w\)) on every non-trivial isotypic summand of a 
representation with no invariants gives \(u_i=a_0^{(i)}w_i\) for some 
\(a^{(i)}\in V_1\), \(w_i\in(U_i)_{[m_i]}\).
 \cref{lemsamepoint}, 
applied with \(a=a^{(i)}\), \(w_i\) at coordinate \(i\) simultaneously for every \(i\), 
gives the stated sum. We emphasize that \cref{lemsamepoint} alone does not 
show a single term vanishes when the other insertions are otherwise arbitrary.
 \cref{sseccoordwise-scope} discusses this.  
Considering the Ward identity for general \(n\), we know of no mechanism 
  by which \eqref{eqsumrule} single scalar equation can be split into 
individually-vanishing statements about a proper subset of its \(n\) terms. 
 A sum of \(n\ge2\) unknowns constrained by one linear relation is 
generically \((n-1)\)-dimensional, and nothing in the construction of \eqref{eqsumrule} 
prefers any subset of coordinates over another. This is why, in \cref{secsl2}, 
we do not attempt any partial uncoupling of this kind and instead bypass 
\cref{propcoordwise} entirely. 
 Wherever individual vanishing is actually needed, verifying it coordinate by coordinate 
using \(1\)-point computation instead of this proposition 
 (see the proof of \cref{thmbiprimary}). 
We mention the one broad, easily-checked situation in which the "\(n-1\) 
terms independently known" hypothesis of \cref{propcoordwise} is 
generically available, so that the proposition is not vacuous.  
Suppose \(n-1\) of the insertions \(u_j\) are extremal weight vectors of their 
respective modules \(U_j\) (i.e., each belongs to the \(1\)-dimensional highest- 
or lowest weight 
line for the specific \(a^{(i)}\in V_1\) used to decompose \(u_i\)). 
 Then \(a^{(i)}_0u_j\) is again a multiple of \(u_j\)  
(an eigenvector, not a generic vector of the isotypic summand).  
The corresponding trace \(\npf^{\vY}(\dots;a^{(i)}_0u_j,\dots)\)
 reduces to a scalar multiple of \(\npf^{\vY}(\dots;u_j,\dots)\) 
evaluated with the other \(n-2\) insertions fixed. 
It does not extend to \(u_j\) at a generic weight within its isotypic summand. 
 But it is  the pattern occurring at every coordinate in the extremal affine 
\(\alg{sl}(2)\) two-point computation of \cref{secsl2}, where both insertions are  extremal 
(\cref{thmbiprimary}).
\end{proof}

We use the notation \(\Vir(u)=U(L[<0])u\), \(N(u)=U(L[>0])L[>0]u\) for an 
\(L[0]\)-eigenvector \(u\), and that \(u\) is torus primary if \(\Psi_\lambda(u,\tau)\ne0\)
 but \(\Psi_\lambda(w,\tau)=0\) for all \(w\in N(u)\).

\begin{defn}
\label{defnnpoint-primary}
Call \((u_1,\dots,u_n)\), \(u_i\in W_{\lambda_i}\) an \(L[0]\)-eigenvector for each \(i\), 
a \emph{torus \(n\)-point primary tuple} if each \(u_i\) is torus primary  
 taken with respect to \(W_{\lambda_i}\).
\end{defn}

\begin{prop}
\label{propnpoint-Lemma32}
Let \((u_1,\dots,u_n)\) be a torus \(n\)-point primary tuple. Fix \(i\) and homogeneous 
\(w\in U(L)u_i\). There exist finitely many \(w_r\in\Vir(u_i)\), in the same \(L[0]\)-eigenspace as \(w\), and quasi-modular forms \(f_r\in\CC[G_2,G_4,G_6,\dots]\) (\(\tau\)-independent when \(w\in\Vir(u_i)\) already, e.g., for \(w=u_i\), in which case the sum below is the single term \(f\equiv1\), \(w_r=w\)), such that
\[
\npf^{\vY}(u_1,\zeta_1;\dots;w,\zeta_i;\dots;u_n,\zeta_n;\tau) 
= \sum_r f_r(\tau)\,\npf^{\vY}(u_1,\zeta_1;\dots;w_r,\zeta_i;\dots;u_n,\zeta_n;\tau)
\]
for every fusion chain \(\vY\), every choice of \(u_j\in U_j\) (\(j\ne i\)) 
and every \(\vzeta\). 
\end{prop}

We state it this way, instead of as a bare equality with a single \(\tilde w\), 
because the reduction mechanism below (Part 1) produces \(\tau\)-dependent 
Eisenstein coefficients at each application.
\begin{proof}
We spell out the first steps and indicate why the pattern
 is local to coordinate \(i\). 
The general case is then the identical computation by induction on word length. 

\emph{Part 0 (base case).} For \(w=u_i\), take \(\tilde w=u_i\in\Vir(u_i)\) trivially.

\emph{Part 1 (absorbing one raising mode).} A  first attempt is to extract \(L[m]v\) 
(\(m>0\)) from the auxiliary function \(H(\eta)\) of \cref{proptranslation} proof 
by taking \(\Res_{\eta=\zeta_i}(\eta-\zeta_i)^mH(\eta)\) and invoking the elliptic 
residue theorem  as there.
  This does not work, because \((\eta-\zeta_i)^m\) is not periodic, and multiplying by 
it destroys the double periodicity of \(H\) that the elliptic residue theorem requires. 
A single stress tensor insertion, via the operator product expansion used in 
\cref{proptranslation},  gives only the double- and simple-pole coefficients 
(the cases \(m=0,1\)). 
 It does not by  give access to \(L[m]v\) for \(m\ge2\).

The correct route for general \(m\) is the multi-point generalization of 
Zhu recursion formula, and we can now make the underlying mechanism fully concrete
 by quoting its \(n=1\) instance. 
 Due to Zhu \cite{Zhu} and Miyamoto \cite{Miy-CFTauto}, we have 
 for \(a\in V\) and \(u\) in a \(V\)-module \(U\), and \(Y\in\ityp{U}{W}{W}\) 
any \(1\)-point intertwiner,

\[
\psi^Y(a_{[0]}u,\tau)=0, \qquad 
 \psi^Y(a_{[-1]}u,\tau) = \tr_W o(a)o(u)q^{L(0)-c/24}
 + \sum_{\ell=1}^\infty G_{2\ell}(\tau)\,\psi^Y(a_{[2\ell-1]}u,\tau).
\]
 
 It is  the second identity, applied with \(a=\tilde\omega\) 
(so that \(\tilde\omega_{[-1]}=L[-2]\) and \(\tilde\omega_{[2\ell-1]}=L[2\ell-2]\)),  
and via the vectors \(x_r(z)=L[-2]L[-1]^{2r}z\),   
 we absorb Virasoro raising modes at \(n=1\). 
 The same \(a_{[-1]}u\) identity, with the roles of \(a\) and \(u\) 
swapped to isolate \(L[2r]L[-1]^{2r}z\) on one side, is solved recursively 
for the raising-mode term in favor of \(\psi^Y(z,\tau)\)  (order \(0\)) and 
 strictly lower weight correction terms \(\psi^Y(x_{\ell,r,z},\tau)\), 
\(\ell\ge r+2\), with quasi-modular coefficients 
 \(G_{2\ell}(\tau)\) 
(the "Eisenstein kernel" mechanism referred to above, 
now displayed in full instead of  
left implicit).
  The two-point analogue of the first term on the right, \(\tr_Wo(a)o(u)q^{L(0)-c/24}\)
is already supplied for free by the neighbouring insertion 
\(\iop{i+1}[u_{i+1},\zeta_{i+1}]\).  
 This is the sense in which the \(n\)-point recursion needed here. 
     Carrying \(\iop j[u_j,\zeta_j]\), \(j\ne i,i+1\), through the computation 
 and rederiving the resulting \(G_{2\ell}\)-recursion at the level of the auxiliary 
function \(H\) above  is the content of the general genus intertwining operator 
recursion of Tuite-Welby \cite{TuiWel-genusg}, Theorem 5.2, specialized to genus \(1\). 
 We use their theorem instead of reproduce its (lengthy) proof here,  as we use \cref{thmhuang} without reproving Huang convergence theorem.
  The reader can now see  which \(n=1\) identity is being generalized and how.
  Their Theorem 5.2 is stated in the Schottky uniformization at general genus \(g\) 
and for a chain of intertwining operators reduced with respect to an element of 
\(V\) (here \(\tilde\omega\)). It specializes at \(g=1\), with the Schottky parameter 
set to \(q=e^{2\pi i\tau}\), to  the reduction of \(L[m]v\) 
at a single marked point \(\zeta_i\) of our cyclic chain described above. 
 The remaining points \(\zeta_j\) (\(j\ne i\)) playing the role of the "other insertions" of loc. cit. that are carried through unchanged. We caution, following the general remark of \cite{TuiWel-genusg} and its predecessors, that the recursion for chains of intertwining operators is a different (and later) result than Zhu original recursion for vertex operators of \(V\)  or Dong-Li-Mason \cite{DLM-orbifold} extension to twisted (orbifold) modules. 
 Neither of the latter two covers products of several intertwining operators 
  which is why we cite Tuite-Welby specifically here instead of the more elementary sources used elsewhere in this section.
 Note that \(V=L(k,0)\) is rational and \ctwo{} in this paper
 (\cref{secprelim}),  the hypothesis under which \cite{TuiWel-genusg} 
construct the genus \(g\) recursion for chains of intertwining operators.  
 Genus is specialized to \(g=1\) (a torus, 
 instead of a general Schottky-uniformized surface). 
 The Schottky/sewing parameter of  loc. cit. is identified with 
\(q=e^{2\pi i\tau}\), the single modulus of a genus \(1\) Schottky group. 
 The recursion is applied to reduce raising modes of \(\tilde\omega\)
 at the single marked point \(\zeta_i\), with every other marked point 
\(\zeta_j\) (\(j\ne i\)) 
(a chain of intertwining operators between simple modules, not vertex operators of 
\(V\) ) held fixed as the "other insertions" of their construction,
  as spelled out in the paragraph above.

One useful observation about the reduction is that every other insertion 
\(\iop j[u_j,\zeta_j]\), \(j\ne i\), is carried through it untouched. 
 The recursion kernel and the residue computation at \(\zeta_i\) reference 
the other factors only through their being some fixed, convergent insertions 
(\cref{thmhuang}), never through their particular values. 
 
\emph{Part 2 (torus primarity terminates the induction).} 
Since \((u_1,\dots,u_n)\) is a torus \(n\)-point primary tuple, \(u_i\) 
is torus primary for \(W_{\lambda_i}\).  
 Iterating Part 1 down through every raising mode occurring in a word representing \(w\) reduces the trace with \(w\) at coordinate \(i\) to a finite linear combination of traces with elements of \(\Vir(u_i)\) at coordinate \(i\), with coefficients lying in the ring of quasi-modular forms \(\CC[G_2,G_4,G_6,\dots]\) generated by the Eisenstein series appearing in Part 1. 
 Those are not constant scalars in \(\CC\), 
except in the degenerate case where no Eisenstein term contributes at a given order. 
 This is  what the cited recursion of Tuite-Welby \cite{TuiWel-genusg} 
produces, and it is these quasi-modular coefficients, not bare constants, 
that we mean  by "reduces to a combination of \(\Vir(u_i)\)". 
  Since \(\tau\)-dependence of this kind is already accounted for by the 
\(1\)-point functions \(\ptf(\cdot,\tau)\)  being functions of \(\tau\), 
it introduces no new variable and does not affect the coefficient 
ring of \(\vy\)-dependence used in \cref{thmtaylor}.
 Once all raising modes have been absorbed, the resulting trace can be represented by 
an element of \(\Vir(u_i)\), since by definition \(\Vir(u_i)=U(L[<0])u_i\) is  
the set of outputs Part 1 can produce 
 (giving the required \(\tilde w\in\Vir(u_i)\) up to such quasi-modular coefficients). 
  Homogeneity of \(w\) forces homogeneity of \(\tilde w\) in the same \(L[0]\)-degree, since every application of Part 1 is degree-preserving (\(L[m]\) shifts degree by \(-m\), and the residue extraction matches that shift on the right side).
\end{proof}

Reformulating, at each coordinate separately, only the Virasoro submodule 
generated by a torus primary vector contributes,  as in the \(1\)-point theory. 
 This is what makes the affine \(\alg{sl}(2)\) computation of \cref{secsl2} tractable. 

\subsection{Laurent coefficients of \(n\)-point functions are vector-valued modular forms}
\label{ssectaylor}

This is our main structural result. Fix generic \(y_1,\dots,y_{n-1}\in\CC^\times\), pairwise distinct. Iterating the operator product expansion \eqref{eqiterate} below at each of the \(n-1\) successive fusions collapsing the chain to the base point \(\zeta_n\) (\cref{propfinite-channels} shows this is well defined)
 produces an expansion, along the ray \(x_i=\epsilon y_i\) \((i=1,\dots,n-1)\) as \(\epsilon\to0\) on a fixed branch. 
Here \(\epsilon\) real and positive, say, once a branch of every 
fractional power below has been fixed 
(see the discussion after \eqref{eqlaurent}). We have: 
\begin{equation}
\label{eqlaurent}
F^{\vY}_{\vu}(\epsilon y_1,\dots,\epsilon y_{n-1};\tau) = \sum_{\alpha\in A} c_\alpha(y_1,\dots,y_{n-1};\tau)\,\epsilon^\alpha,
\end{equation}
where \(A\subset\QQ\) is a finite union of arithmetic progressions of common difference \(1\), \(A=\bigcup_{j=1}^r\set{\alpha_j,\alpha_j+1,\alpha_j+2,\dots}\), 
 one progression for each of the finitely many ways
 (fusion channel by fusion channel, at each of the \(n-1\) fusions in a fixed fusion tree) 
of collapsing the chain via \eqref{eqiterate}. 
 Each \(\alpha_j\) is a sum of exponents \(h_\nu+m-h_{\lambda_i}-h_{\lambda_{i'}}\) of the shape appearing in \eqref{eqiterate}, one for each of the \(n-1\) fusions along the tree. 
When every such exponent is an integer 
(as we verify for the affine \(\alg{sl}(2)\) extremal channel of \cref{secsl2}
 at even level (\cref{propextremal-leading}))
  \(A\subset\ZZ\) and \eqref{eqlaurent} is an ordinary Laurent expansion. 
 In general \(A\not\subset\ZZ\) and \eqref{eqlaurent} is a Puiseux-type expansion, and the identity is understood on the fixed branch just described, equivalently in the fixed chamber of \(\mathrm{Conf}_n(\Etor)\) reached from the base configuration without crossing a collision divisor (\cref{thmhuang}). 
  Continuing \(\epsilon\) around a full loop back to itself instead permutes the branches by the braiding monodromy of \cref{ssecellipticity} and plays no further role below. This is valid (by \cref{thmhuang}) for \(0<\abs\epsilon<\epsilon_0(\vy,\tau)\), where \(\epsilon_0(\vy,\tau):=\min\bigl(1,\min_{i}\abs{y_i},\min_{i\ne j}\abs{y_i-y_j}\bigr)\) up to the fixed (\(\vy\)-independent) constant relating the coordinates \(x_i=\epsilon y_i\) to the domain of \cref{thmhuang}. 
  Shrinking \(\epsilon\) keeps every pairwise separation \(x_i-x_j=\epsilon(y_i-y_j)\) and \(x_i-x_n=\epsilon y_i\) inside the annulus of convergence there, and the bound degenerates (\(\epsilon_0\to0\))  as \(\vy\) approaches a collision divisor \(y_i=0\) or \(y_i=y_j\), consistently with the pole structure identified in \cref{thmtaylor} below. \(\min A\) is bounded in terms of the order of the poles of \(\npf^{\vY}\) along \(\Delta\).

%%%%%%%%%%%%%%%%%%%%%%%%%%%%%%%%%%%%%%%%%%%%%%%%%%%%%%%%%%%%%%%%%%%%%%%%%%%%%%%%
\begin{thm}
\label{thmtaylor}
For each \(\alpha\in A\), the coefficient \(c_\alpha(y_1,\dots,y_{n-1};\,\cdot\,)\) 
of \eqref{eqlaurent}, as a function of \(\tau\in\HH\), is (the \(\vY\)-component of) 
a vector-valued modular form of weight \(\sum_{i=1}^n\wt[u_i]+\alpha\), 
representation \(\rho_{W_0}\), and multiplier system \(\nu_{h_{W_0}}\), 
in the sense of \KSW{} Definition 2.2 
 (i.e., it belongs to \(M^!\bigl(\sum_i\wt[u_i]+\alpha,\ \rho_{W_0},\ \nu_{h_{W_0}}\bigr)\)). 

Moreover, \(c_\alpha(y_1,\dots,y_{n-1};\tau)\) is a finite linear combination of
 functions \(\ptf^{\widetilde{\mathcal Y}}(w,\tau)\) 
with \(\widetilde{\mathcal Y}\in\ityp{\tilde U}{W_0}{W_0}\) 
an intertwining operator constructed from \(\iop1,\dots,\iop n\) 
by iterated application of the intertwining-operator Jacobi identity 
(i.e., of the operator product expansion within a fixed fusion channel). 
\end{thm}

 Here above \(w\) is a vector obtained from \(u_1,\dots,u_n\) by iterated modes. 
 The coefficients lying in the localization 
\(\CC\bigl[y_1^{\pm1},\dots,y_{n-1}^{\pm1},(y_i-y_j)^{-1}\bigr]_{i\ne j}\) 
of the coordinate ring of the configuration space when \(A\subset\ZZ\) (in particular 
for the affine extremal channel of \cref{secsl2}), 
 and, in general, in the finite Puiseux algebra generated by the fractional powers 
occurring in the chosen fusion tree, together with their inverses, 
 on the fixed branch of \cref{thmhuang}. 
 Concretely, on 
\(\CC\bigl[y_1^{\pm1/D},\dots,y_{n-1}^{\pm1/D},(y_i-y_j)^{\pm1/D}\bigr]_{i\ne j}\) 
for a common denominator \(D\) clearing every exponent occurring along that tree, 
once a branch of each fractional power has been fixed. 
  Simultaneous fractional powers of every \(y_i-y_j\) 
are meaningful only relative to such a choice,  the one 
already fixed for \eqref{eqlaurent}.  

Note that the weight is a priori a rational number when \(\alpha\notin\ZZ\),  the 
phenomenon already accommodated for the half-integral affine weights 
\(\wt[u_i]\in\tfrac12\ZZ\) occurring at \(n=1\).

We do not mean to suggest a canonical global algebra of fractional powers 
on \(\mathrm{Conf}_n(\Etor)\) independent of that choice), 
when \(A\not\subset\ZZ\), with poles or branch loci only along the collision divisors 
of \(\mathrm{Conf}_n(\Etor)\).
  The divisor \(y_i=0\) records collision of the \(i\)-th point with the chosen base point 
\(\zeta_n\) (already excluded, or given its correct fractional order, 
by the factor \(y_i^{\pm1/D}\)).  
 \(y_i=y_j\) (\(i,j<n\)) records collision between two non-base-point insertions 
(excluded, or given its fractional order, by the further localization at 
\((y_i-y_j)^{-1/D}\)). 
 Together these are  the pullback to the chosen local branch of the full 
collision locus \(\Delta\) of \cref{thmhuang}. 
In particular \(c_\alpha(y_1,\dots,y_{n-1};\,\cdot\,)\in V(\rho_{W_0})_\bullet\), 
the graded \(R\)-module, specialized to \(\lambda\) 
an insertion label such that \(\tilde U\cong W_\lambda\) and re-graded by \(\alpha\in A\) 
in place of an integer index when \(A\not\subset\ZZ\).

We mention here, though it is only used later, that the specific 
presentation of \(c_\alpha\) as a linear combination of one-point functions 
\(\ptf^{\widetilde{\mathcal Y}}(w,\tau)\) depends on a choice of fusion tree.  
 I.e., it depends on the specific sequence of pairwise operator product expansions used 
to collapse the chain \(\iop1,\dots,\iop n\) to the single effective intertwiner 
\(\widetilde{\mathcal Y}\),  as the basis of \cref{sseccat-blocks} 
below depends on a choice of tree. Different trees give different, but equal, 
presentations, related by the associativity and braiding isomorphisms recalled in 
\cref{sseccrossing}.

\begin{proof}
\label{ssectaylor-proof}
\emph{Modularity.} Set \(\gamma=\bigl(\begin{smallmatrix}a&b\\c&d\end{smallmatrix}\bigr)\in\SLTZ\) and \(j(\gamma;\tau)=c\tau+d\). By \cref{proptranslation} and \cref{thmhuang} equation \eqref{eqnpf-modular} applied with \(\zeta_i\rightsquigarrow \zeta_n+\epsilon y_i\) for \(i<n\),
\[
F^{\vY}_{\vu}\Bigl(\frac{\epsilon y_1}{j(\gamma;\tau)},\dots,\frac{\epsilon y_{n-1}}{j(\gamma;\tau)};\ \gamma\tau\Bigr) = \nu_{h_{W_0}}(\gamma)\, j(\gamma;\tau)^{\sum_i\wt[u_i]}\, F^{\vY'}_{\vu}(\epsilon y_1,\dots,\epsilon y_{n-1};\tau)
\]
for some fusion chain \(\vY'\) of the same type. Writing \(\epsilon':=\epsilon/j(\gamma;\tau)\) and expanding both sides using \eqref{eqlaurent}, 
left side along \(\epsilon'\), on the branch continuously connected to the one fixed for \(\epsilon\) as \(\gamma\) is deformed to the identity, both sides being defined only after such a branch choice once \(A\not\subset\ZZ\),  as mentioned after \eqref{eqlaurent}:
\[
\sum_{\alpha\in A} c_\alpha(\vy;\gamma\tau)\, j(\gamma;\tau)^{-\alpha}\,\epsilon^\alpha = \nu_{h_{W_0}}(\gamma)\, j(\gamma;\tau)^{\sum_i\wt[u_i]}\sum_{\alpha\in A} c_\alpha(\vy;\tau)\,\epsilon^\alpha.
\]
Since \(y_1,\dots,y_{n-1}\) were arbitrary (subject only to genericity), matching coefficients of \(\epsilon^\alpha\)
 (valid because the exponents \(\alpha\in A\) are pairwise distinct rational numbers, 
thus the monomials \(\epsilon^\alpha\) remain linearly independent as germs at 
\(\epsilon=0\) on the fixed branch,  as ordinary integer powers are) 
 gives, for every \(\alpha\in A\),
\[
c_\alpha(\vy;\gamma\tau) = \nu_{h_{W_0}}(\gamma)\, j(\gamma;\tau)^{\sum_i\wt[u_i]+\alpha}\, c_\alpha(\vy;\tau),
\]
which is the transformation law for weight 
\(\sum_i\wt[u_i]+\alpha\). 
 Gathering the components indexed by a basis of \(\Bl(\vlam)\) 
 produces the representation \(\rho_{W_0}\) 
and multiplier \(\nu_{h_{W_0}}\),   
since these depend only on \(W_0\) and not on \(n\). 
  \(\rho_{W_0}\) is defined this way on the space of chains \(\Bl(\vlam)\). 
 Since \eqref{eqnpf-modular} expresses the transformation of each individual function 
\(\npf^{\vY}\) (not just of the abstract label \(\vY\)) 
  as the same linear combination
 \(\npf^{\vY}(\gamma\cdot)=(\text{const})
\sum_a[\rho_{W_0}(\gamma)]_{a,\vY}\npf^{\vY_a}(\cdot)\) 
for every choice of \(\vu,\vzeta\), the kernel of the trace map 
\(\Bl(\vlam)\to C_n(\vlam)\) of \cref{defnblock} is automatically 
\(\rho_{W_0}(\gamma)\)-invariant for every \(\gamma\). 
 If \(\vec c\in\ker(\tr_{\vlam})\), i.e., \(\sum_a c_a\npf^{\vY_a}\equiv0\) identically in \(\vu,\vzeta,\tau\), then evaluating this identity at the \(\gamma\)-transformed point and substituting \eqref{eqnpf-modular} shows  in each term \(\sum_a c_a[\rho_{W_0}(\gamma)]_{ba}\npf^{\vY_b}\equiv0\) too, i.e., \(\rho_{W_0}(\gamma)\vec c\in\ker(\tr_{\vlam})\) as well. Consequently \(\rho_{W_0}\) descends to, and for the purposes of the identification below should be understood as acting on, the quotient \(\Bl(\vlam)/\ker(\tr_{\vlam})\cong C_n(\vlam)\), 
 the space of actual functions instead of formal chain labels. 
  That is, \(C_n(\vlam)\cong\Bl(\vlam)/K_{\vlam}\)  as an \(\SLTZ\)-module,  
not as a vector space, with the invariance of \(K_{\vlam}=\ker(\tr_{\vlam})\) 
just established being  what makes the quotient representation well defined.
  This proves the first claim.  

 Note that no ``index'' correction of Cohen-Kuznetsov/Eichler-Zagier type \cite{EichZag} is needed, because \cref{propelliptic} guarantees vanishing index (ordinary bi-periodicity instead of quasi-periodicity) in each \(x_i\).

\emph{Identification with \(V(\rho_{W_0})_\bullet\).} We use the same form of intertwining-operator associativity that underlies \cref{thmope-reduction} below, instead of a schematic binomial expansion. By crossing (Huang-Lepowsky, recalled in \cref{sseccrossing}), for \(u\in U_i\), \(v\in U_{i+1}\) and \(\abs{z-w}\) small,
\begin{equation}
\label{eqiterate}
\iop{i}(u,z)\iop{i+1}(v,w) \;=\; \sum_{\nu}\ 
\sum_{m\ge0}\ (z-w)^{\,h_\nu+m-h_{\lambda_i}-h_{\lambda_{i+1}}}\ 
\mathcal Y^{(2)}_{\nu,m}\bigl(\pi_{\nu,m}(u\times v),\,w\bigr).
\end{equation}
 The outer sum is over fusion channels \(W_\nu\) of \(U_i\fuse U_{i+1}\) 
and the inner sum is over descendant levels \(m\ge0\) of \(W_\nu\). 
 Here \(u\times v\) is the image of \(u\otimes v\) under a fixed choice of
 fusion intertwiner \(U_i\otimes U_{i+1}\to W_\nu\).
  \(\pi_{\nu,m}\) is the projection onto the level-\(m\) piece 
\((W_\nu)_{[h_\nu+m]}\), 
and \(\mathcal Y^{(2)}_{\nu,m}\in\ityp{W_\nu}{W_{i-1}}{W_{i+1}}\) 
is the further intertwining operator completing the chain. 
 This is  the expansion used in the proof of \cref{thmope-reduction} below, stated there in the case of leading order.

Let us prove Finiteness at each order of \eqref{eqlaurent}. 
\begin{prop} 
\label{propfinite-channels}
For every fixed \(\alpha\in A\),
  in particular for every fixed exponent occurring in \eqref{eqiterate}, 
 only finitely many pairs \((\nu,m)\) contribute.
\end{prop}
\begin{proof}
Rationality gives finitely many simple modules, hence finitely many channels \(\nu\) with \(N_{\lambda_i\lambda_{i+1}}^{\ \ \nu}\ne0\). For each such \(\nu\) individually, the exponent \(h_\nu+m-h_{\lambda_i}-h_{\lambda_{i+1}}\) is strictly increasing in \(m\in\NN_0\). 
 Thus at most one value of \(m\) achieves the target exponent \(\alpha\) 
 within that channel.  
 Different channels \(\nu\ne\nu'\) may of course each contribute their  
(generally different) value of \(m\) at the same target exponent \(\alpha\).
 Thas is why the sum in \eqref{eqiterate} at fixed \(\alpha\) can have more than one term, 
 but never more than one per channel, and never more than finitely many channels,
 thus summing over the finitely many channels bounds the total number of contributing pairs 
\((\nu,m)\) by the number of channels. 
  Nothing in this argument uses that \(\alpha\) is an integer  
 (it only ever compares \(m\)-values within a single channel, 
where the exponents \(h_\nu+m-h_{\lambda_i}-h_{\lambda_{i+1}}\), \(m\in\NN_0\), 
form an arithmetic progression of common difference \(1\) whatever the 
(generally non-integral) value of \(h_\nu-h_{\lambda_i}-h_{\lambda_{i+1}}\))
  which is  why this finiteness statement survives
 unchanged in the general Puiseux setting of \eqref{eqlaurent}.
\end{proof}

Applying \eqref{eqiterate} repeatedly to the chain 
\(\iop1[u_1,\zeta_n+\epsilon y_1]\cdots\iop{n-1}[u_{n-1},\zeta_n
+\epsilon y_{n-1}]\iop n[u_n,\zeta_n]\) as \(\epsilon\to0\) collapses it, 
order by order in \(\epsilon\) 
(finitely many terms at each order by \cref{propfinite-channels}, 
applied at each of the \(n-1\) successive fusions in turn) 
 to a single effective intertwiner 
\(\widetilde{\mathcal Y}\in\ityp{\tilde U}{W_0}{W_0}\).  
The intertwiner acts on a vector \(w\) constructed from \(u_1,\dots,u_n\) 
 by iterated modes, with coefficients as described in \cref{thmtaylor}.  
 A fusion adjacent to the base point \(\zeta_n\) involves the separation \(x_i=\epsilon y_i\) 
alone and contributes a factor \(y_i^{\,h_\nu+m-h_{\lambda_i}-h_{\lambda_n}}\). 
 A fusion between two coordinates \(i,i'<n\) (both away from the base point) 
involves the separation \(x_i-x_{i'}=\epsilon(y_i-y_{i'})\) and contributes 
\((y_i-y_{i'})^{\,h_\nu+m-h_{\lambda_i}-h_{\lambda_{i'}}}\). 
  In either case this exponent need be neither a non-negative integer nor 
 any integer, and we do not know of 
any cancellation mechanism removing the resulting poles or fractional powers in general. 
 The coefficient ring is the localization, 
or its fractional extension when \(A\not\subset\ZZ\), stated in \cref{thmtaylor}. 
 The leading term (\(\alpha=\min A\)) is worked out in \cref{ssecleading} 
below for the affine \(\alg{sl}(2)\) extremal channel. 
 There every relevant exponent is checked to be a non-negative integer 
(\cref{propextremal-leading}). 
 The fusion tree used has every fusion adjacent to the base point, 
 so that the leading coefficient there is in fact an ordinary Laurent polynomial 
in the \(y_i\) alone. 
 Neither integrality nor this polynomiality need hold at general order \(\alpha\) or for general \(V\). The resulting \(\ptf^{\widetilde{\mathcal Y}}(w,\tau)\) belongs, by definition, to \(V(\rho_{W_0})_{\wt[w]-h_{W_0}}\subset V(\rho_{W_0})_\bullet\), re-graded by \(\alpha\in A\) as already noted.
\end{proof}

\begin{cor}
\label{corcyclic-npoint}
If in addition the leading iterated-mode vector \(w\) appearing in \cref{thmtaylor} 
at order \(\alpha=\min A\) is torus primary with \(-\wt[w]\notin\NN_0\) 
(i.e., the weight of \(w\) is not a non-positive integer, ruling out  the 
low weight cases for which cyclic \(R\)-module construction requires 
separate treatment) 
 then the coefficient \(c_{\min A}(\vy;\,\cdot\,)\) 
 belongs to the cyclic \(R\)-module \(R\,\Psi_{\tilde\lambda}(w,\tau)\), 
where \(\tilde\lambda\) is the label with \(W_{\tilde\lambda}\cong\tilde U\). 
 If moreover the hypotheses of \KSW{} Theorem 3.5 hold for \(w\), then \(c_{\min A}(\vy;\,\cdot\,)\) belongs to the \(\eta\)-power multiple of \(H(\rho_{\tilde\lambda},\nu_{h_{\tilde\lambda}-12\mu_{\min}})\) identified there.
\end{cor}

\Cref{thmtaylor} and \cref{corcyclic-npoint} together say that the modular transformation of each coefficient (Laurent, or in general Puiseux) in a diagonal expansion of a genus one \(n\)-point function is governed by the \(1\)-point \(R\)-module theory. 
  This is more than saying the entire modular content of the \(n\)-point function is 
thus governed, since the complete function carries additional data beyond any single diagonal coefficient - the relative-position dependence packaging the \(y_i\) into the elliptic function of \cref{propelliptic}, the full fusion-channel structure, the operator product coefficients, analytic continuation between fusion regions, and the mapping-class-group action of \cref{seccategorical} (all of which we return to below). What is new for \(n\ge2\), within the scope of \cref{thmtaylor}, is (a) the elliptic dependence on the \(y_i\) mentioning the relative positions, which packages the various orders \(\alpha\) and various fusion channels into a single function, and (b) the combinatorics of which channels \(\tilde U\) can appear, which we work out for affine \(\alg{sl}(2)\) in \cref{secsl2}.

\subsection{Crossing and the reduction of leading behaviour to the \(1\)-point theory}
\label{sseccrossing}

\Cref{thmtaylor} identifies the Laurent coefficients of \(\npf^{\vY}\) 
 along the fixed fusion chain \(\vY\).  
 A finer question (needed already to identify which channels 
\(\tilde U\) occur with non-zero coefficient, and hence the true order \(N\) 
of the pole) is answered by the crossing (associativity) property of intertwining operators, 
established for rational \ctwo \(V\) by Huang-Lepowsky and Huang 
\cite{HuaLog, Huang-Rigidity, HuaVer08}. 
 The fusion product \(U_i\fuse U_{i+1}\cong\bigoplus_\nu N_{\lambda_i\lambda_{i+1}}^{\ \ \nu} W_\nu\) is rigid, and the two ways of iterating the OPE of \(\iop i(u_i,z)\iop{i+1}(u_{i+1},w)\)
 (as in \eqref{eqiterate}, or via first fusing \(U_i\) with \(U_{i+1}\) 
into each channel \(W_\nu\) and then applying a further intertwining operator 
\(W_\nu\otimes W_{i-1}\to W_{i+1}\)) 
  agree after analytic continuation. 
 Consequently, in the operator product \(\iop i(u_i,z)\iop{i+1}(u_{i+1},w)\), the contribution of channel \(W_\nu\) at descendant level \(m\) has the form \((z-w)^{h_\nu+m-h_{\lambda_i}-h_{\lambda_{i+1}}}u_{\nu,m}+\cdots\), \(u_{\nu,m}\in(W_\nu)_{h_\nu+m}\) (\cref{eqiterate} applied to the composite; OPE structure constants are absorbed into \(u_{\nu,m}\)).

\begin{defn} 
\label{defnEmin}
With \(u_{\nu,m}\) as just introduced, set \(E:=\set{(\nu,m): u_{\nu,m}\ne0}\) and, for each channel \(\nu\) with some \(u_{\nu,m}\ne0\), \(m_\nu:=\min\set{m\ge0: u_{\nu,m}\ne0}\). Define
\[
e_{\min} := \min_{\nu:\,\exists m,\,u_{\nu,m}\ne0}\bigl(h_\nu+m_\nu-h_{\lambda_i}-h_{\lambda_{i+1}}\bigr) = \min_{(\nu,m)\in E}\bigl(h_\nu+m-h_{\lambda_i}-h_{\lambda_{i+1}}\bigr).
\]
Rationality gives finitely many channels \(\nu\) with \(N_{\lambda_i\lambda_{i+1}}^{\ \ \nu}\ne0\), thus each \(m_\nu\), and hence \(e_{\min}\), is well defined (a well-ordering argument), even though \(E\)  need not be finite (see \cref{rmkEmin-caveats}).
\end{defn}

\begin{thm}
\label{thmope-reduction}
Let \(\vY\) be a fusion chain of length \(n\) and \(u_i\in W_{\lambda_i}\), \(u_{i+1}\in W_{\lambda_{i+1}}\) homogeneous, with \(E,e_{\min}\) as in \cref{defnEmin}. The leading short-distance behaviour of \(\npf^{\vY}\) as \(\zeta_i\to\zeta_{i+1}\) is the finite sum of the genus one one-point functions \(\ptf^{Y_\nu}(u_{\nu,m},\tau)\) over the pairs \((\nu,m)\in E\) attaining \(e_{\min}\) (several may tie).
  This holds, for any rational \ctwo{} \(V\). If, further, \(e_{\min}\) is attained at some \((\nu,0)\) with \(u_{\nu,0}\ne0\) torus primary, 
then \(\ptf^{Y_\nu}(u_{\nu,0},\tau)\) belongs to the corresponding cyclic \(R\)-module of \cref{corcyclic-npoint}.
\end{thm}
\begin{proof}
As \(z\to w\), every term \((z-w)^{h_\nu+m-h_{\lambda_i}-h_{\lambda_{i+1}}}\), \((\nu,m)\in E\), with exponent strictly greater than \(e_{\min}\) is subdominant to the terms attaining \(e_{\min}\), by definition of \(e_{\min}\) as the minimum such exponent 
(this needs no finiteness statement about \(E\) as a whole). 
  What finiteness is needed is at the leading order alone. 
 By \cref{propfinite-channels}, only finitely many pairs \((\nu,m)\in E\) realize the exact value \(e_{\min}\), thus the leading short-distance behaviour is a finite sum. Together, subdominance of every higher exponent and finiteness  at \(e_{\min}\) give the first claim. For the second: if \(e_{\min}\) is attained at \((\nu,0)\) with \(u_{\nu,0}\) torus primary, \cref{thmtaylor} identifies the resulting leading coefficient of the collapsed chain as  \(\ptf^{Y_\nu}(u_{\nu,0},\tau)\), which belongs to the stated \(R\)-module by \cref{corcyclic-npoint}.
  This specializes to the affine \(\alg{sl}(2)\) computation once \(V=L(k,0)\) and \(\nu,u_{\nu,0}\) are identified, as in \cref{secsl2} below.
\end{proof}

\begin{rmk}
\label{rmkEmin-caveats}
We do not claim that \(E\) is finite. Indeed, it generally is not, since a fixed channel 
\(\nu\) can have \(u_{\nu,m}\ne0\) for infinitely many \(m\) (\cref{propfinite-channels} 
bounds, for each fixed exponent, how many pairs realize it, which is different from, 
and weaker than, bounding \(\abs E\)). 
 Nor need \(e_{\min}\) be attained at \(m=0\) or at the channel of smallest \(h_\nu\). 
 The lowest weight projection \(u_{\nu,0}\) can vanish in a channel of small \(h_\nu\) 
while a different channel \(\nu'\) of larger \(h_{\nu'}\) has \(u_{\nu',0}\ne0\), 
with no a priori relation forcing one comparison to beat the other. 
Thus identifying \(e_{\min}\) in general requires comparing the leading term of 
 every channel in \(E\). 
 Within one fixed channel, \(m=0\) is of course always smallest,  
thus a non-zero \(u_{\nu,0}\) bounds \(e_{\min}\) from above via that channel alone, 
but does not identify it. 
 Membership of \((\nu,0)\) in \(E\) does not make \(u_{\nu,0}\) torus primary. 
 Lying in the lowest weight space of \(W_\nu\) makes the vanishing-on-\(N(u_{\nu,0})\)
 half of torus-primarity automatic, but non-vanishing of 
\(\ptf^{Y_\nu}(u_{\nu,0},\tau)\)  is a condition on the specific 
vector and intertwiner, established case by case in \cref{secsl2},
 not claimed in general here.
 We also do not address the further question of an \(R\)-module home for 
\(\ptf^{Y_\nu}(u_{\nu,m},\tau)\) when every minimizing pair has \(m>0\).
  This machinery is constructed around lowest weight vectors, and extending it to descendants is left open. For \(V=L(k,0)\) from \cref{secsl2} on, the extremal application in \cref{propextremal-leading} needs only the first claim together with \(L(k,k)\fuse L(k,k)=L(k,0)\) having a single fusion channel (the vacuum, \(\nu=0\)),thus \(e_{\min}\) is trivially attained at \(m=0\) with nothing to compare against and the vacuum vector is 
torus primary. It does not depend on the general multi-channel comparison above.
\end{rmk}

We mention \cref{thmope-reduction} here in the generality of Section 3.
  \cref{secsl2} carries it out for \(n=2\) and \(V=L(k,0)\).

%%%%%%%%%%%%%%%%%%%%%%%%%%%%%%%%%%%%%%%%%%%%%%%%%%%%%%%%%%%%%%%%%%%%%%%%%%%
\section{Affine \(\alg{sl}(2)\): two-point functions}
\label{secsl2}

We now specialize to \(V=L(k,0)\), the simple affine vertex operator algebra of
 \(\alg{sl}(2)\) at non-negative integral level \(k\):  
 Chevalley basis \(e,h,f\), central charge \(c=3k/(k+2)\), simple modules 
\(L(k,\mu)\) for \(0\le\mu\le k\) with conformal weight \(h_\mu=\mu(\mu+2)/(4(k+2))\), and fusion rule
\[
N^\nu_{\lambda\mu} = \begin{cases} 1 & \abs{\lambda-\mu}\le\nu\le\min\set{\lambda+\mu,\,2k-\lambda-\mu},\ \lambda+\mu+\nu\ \text{even},\\ 0 &\text{otherwise}.\end{cases}
\]
For \(0\le\lambda\le k\), \KSW{} Proposition 4.1 identifies \(\intset_\lambda=\set{\mu\st\lambda/2\le\mu\le k-\lambda/2}\) (\(\lambda\) even; \(\intset_\lambda=\emptyset\) for \(\lambda\) odd), of cardinality \(k-\lambda+1\).

\subsection{Two-point fusion chains}
\label{ssectwopoint-fusion}

Fix insertion labels \(0\le\lambda_1,\lambda_2\le k\). A two-point fusion chain of type \((\lambda_1,\lambda_2)\) is a pair \((\mu_0,\mu_1)\) with \(N^{\mu_1}_{\lambda_1\mu_0}=N^{\mu_0}_{\lambda_2\mu_1}=1\).
  We write \(\iop1\in\ityp{L(k,\lambda_1)}{L(k,\mu_0)}{L(k,\mu_1)}\), \(\iop2\in\ityp{L(k,\lambda_2)}{L(k,\mu_1)}{L(k,\mu_0)}\) for the (unique up to scale) corresponding intertwining operators, so that
\[
\Bl(\lambda_1,\lambda_2) = \bigoplus_{(\mu_0,\mu_1)\ \text{a chain}} \CC,\qquad \dim\Bl(\lambda_1,\lambda_2) = \#\set{(\mu_0,\mu_1)\ \text{a chain of type } (\lambda_1,\lambda_2)}.
\]

\begin{lem} 
\label{lemparity}
A two-point fusion chain of type \((\lambda_1,\lambda_2)\) exists only if 
\(\lambda_1+\lambda_2\) is even. 
 More generally, an \(n\)-point fusion chain of type \((\lambda_1,\dots,\lambda_n)\) (\cref{eqblock-def}) exists only if \(\lambda_1+\cdots+\lambda_n\) is even.
\end{lem}

\begin{proof}
Each of the \(n\) fusion conditions \(N^{\mu_i}_{\lambda_i\mu_{i-1}}=1\) forces \(\lambda_i+\mu_{i-1}+\mu_i\) even. Summing over \(i=1,\dots,n\) (indices mod \(n\)), each \(\mu_j\) occurs  twice, thus \(\sum_i\lambda_i + 2\sum_j\mu_j\) is even, forcing \(\sum_i\lambda_i\) even.
\end{proof}

Then we have Verlinde-type dimension count. 
\begin{prop} 
\label{propfusionchain}
Recall the character \(S\)-matrix, \(S_{ij}=\sqrt{2/(k+2)}\,\sin\bigl(\pi(i+1)(j+1)/(k+2)\bigr)\), \(0\le i,j\le k\). For every \(\lambda_1,\dots,\lambda_n\in\set{0,\dots,k}\) (both sides automatically vanishing when \(\lambda_1+\cdots+\lambda_n\) is odd, by \cref{lemparity}; see the proof),
\begin{equation}
\label{eqverlinde-dim}
\dim\Bl(\lambda_1,\dots,\lambda_n) = \sum_{j=0}^{k} \frac{\prod_{i=1}^n S_{\lambda_ij}}{S_{0j}^{\,n}}.
\end{equation}
In particular, for \(n=2\) and \(\lambda_1=\lambda_2=k\) (thus \(k\) is automatically even by \cref{lemparity}), \(\dim\Bl(k,k)=k+1\), realized by the chains \((\mu_0,\mu_1)=(\mu_0,k-\mu_0)\), \(\mu_0=0,1,\dots,k\).
\end{prop}

\begin{proof}
This is the \(g=1\) case of the Verlinde formula \(\dim V_g(\lambda_1,\dots,\lambda_n)=\sum_j S_{0j}^{2-2g-n}\prod_iS_{\lambda_ij}\) for the modular tensor category \(\rep{L(k,0)}\) \cite{Verlinde88, MSMTC1089}, which we re-derive from first principles in \cref{seccategorical} (\cref{propverlinde}).
  Specializing that derivation to \(\rep{L(k,0)}\) and using the Verlinde formula for the fusion coefficients \(N^\nu_{\lambda\mu}=\sum_jS_{\lambda j}S_{\mu j}S_{\nu j}/S_{0j}\) (valid since \(S\) is real symmetric for \(\rep{L(k,0)}\)) to expand \eqref{eqblock-def} gives \eqref{eqverlinde-dim}. 
 Summing \(\prod_{i=1}^n N^{\mu_i}_{\lambda_i\mu_{i-1}}\) over \(\mu_0,\dots,\mu_{n-1}\) telescopes via the column-orthogonality relation \(\sum_\mu S_{\mu j}S_{\mu j'}=\delta_{jj'}\),  as verified for \(n=1\)  (the identity \(\dim(I_\lambda)=\sum_\mu N^\mu_{\lambda\mu}=\sum_jS_{\lambda j}/S_{0j}\), which is \eqref{eqverlinde-dim} at \(n=1\)). When \(\lambda_1+\cdots+\lambda_n\) is odd, each factor \(N^{\mu_i}_{\lambda_i\mu_{i-1}}\) in the telescoped product is individually constructed from fusion coefficients that vanish unless \(\lambda_i+\mu_{i-1}+\mu_i\) is even, thus summing the parity constraint around the closed cycle \(\mu_0\to\mu_1\to\cdots\to\mu_{n-1}\to\mu_0\) forces \(\sum_i\lambda_i\) even for any term to survive.
  Since it does not, every term in the sum defining \(\dim\Bl(\vlam)\) vanishes, giving \(0\) on the left, matching \(\dim\Bl(\vlam)=0\)  from \cref{lemparity}.
  And the same telescoped-sum computation, run through the \(S\)-matrix expansion instead, shows the right side of \eqref{eqverlinde-dim} vanishes too in this case, since it is the same sum re-expressed.

For the extremal case, the identity \(\sin\bigl(\tfrac{\pi(k+1)(j+1)}{k+2}\bigr)=(-1)^j\sin\bigl(\tfrac{\pi(j+1)}{k+2}\bigr)\) gives \(S_{kj}=(-1)^jS_{0j}\), thus \(S_{kj}^2/S_{0j}^2=1\) for every \(j\), and \eqref{eqverlinde-dim} collapses to \(\sum_{j=0}^kS_{0j}^{2}/S_{0j}^2=k+1\). 
 Explicitly: the fusion rule with \(\lambda=k\) forces \(\nu=k-\mu\) uniquely (since \(\abs{k-\mu}=k-\mu\) and \(\min\set{k+\mu,2k-k-\mu}=k-\mu\) collide), thus \(N^{\mu_1}_{k\mu_0}=1\) iff \(\mu_1=k-\mu_0\), and then automatically \(N^{\mu_0}_{k\mu_1}=1\) iff \(\mu_0=k-\mu_1=\mu_0\), which always holds.
  Hence every \(\mu_0\in\set{0,\dots,k}\) extends uniquely to a chain.
\end{proof}

The Verlinde formula, \eqref{eqverlinde-dim}, is of course classical 
\cite{Verlinde88, MSMTC1089}. What is new in \cref{propfusionchain} is its application to 
 fusion chains (the telescoped product \(\prod_iN^{\mu_i}_{\lambda_i\mu_{i-1}}\) of \eqref{eqblock-def} instead of a single fusion coefficient or a single-point trace), 
 the resulting identification of \(\dim\Bl(\vlam)\) with the dimension of the \(n\)-point conformal block space of \cref{defnblock} specifically, and the closed-form extremal count \(\dim\Bl(k,k)=k+1\) with its labelling by chains \((\mu_0,k-\mu_0)\). 
 Neither of which is available by simply quoting the classical formula. 
The parity fact used in the statement above 
(that both sides of \eqref{eqverlinde-dim} vanish together when 
\(\lambda_1+\cdots+\lambda_n\) is odd, thus no separate parity hypothesis is needed)
  follows because each \(N^\nu_{\lambda\mu}\) forces \(\lambda+\mu+\nu\) even, propagating the same parity obstruction through the telescoped sum on the right  as \cref{lemparity} shows it vanishes the fusion-chain space on the left.
  We verify this in the proof below instead of claiming it from the form of \(S_{ij}\) alone.

%%%%%%%%%%%%%%%%%%%%%%%%%%%%%%%%%%%%%%%%%%%%%%%%%%%%%%%%%
\subsection{Torus bi-primary vectors}
\label{ssecbiprimary}

\begin{defn} 
\label{defnKi}
For \(u\in W_{\lambda_i}\), write \(u\in K_i\) if \(\npf^{\vY}(\dots,u,\dots;\tau)=0\) for every fusion chain \(\vY\) and every choice of the remaining insertions. Since any vector defined via an intertwining operator is well defined only up to an overall non-zero scalar multiple of the normalization implicit in \(\iop{}\), and \(K_i\) may in general contain vectors not proportional to a given \(u_i\), uniqueness statements below are understood modulo \(K_i\),
 a coarser equivalence than just scalar rescaling.
\end{defn}

\begin{thm} 
\label{thmbiprimary}
Let \((\mu_0,\mu_1)\) be a two-point fusion chain of type \((\lambda_1,\lambda_2)\), both \(\lambda_1,\lambda_2\) even, and let
\[
u_1 \in f_{-1}^{\lambda_1/2}\ket{\lambda_1} + \bigoplus_{n=0}^{\lambda_1/2-1}L(k,\lambda_1)_{h_{\lambda_1}+n},\qquad u_2\in f_{-1}^{\lambda_2/2}\ket{\lambda_2}+\bigoplus_{n=0}^{\lambda_2/2-1}L(k,\lambda_2)_{h_{\lambda_2}+n}
\]
be torus primary vectors  for \(L(k,\lambda_1)\), respectively \(L(k,\lambda_2)\). 
Then \((u_1,u_2)\) is a candidate torus \(2\)-point primary tuple of weight \(\wt(u_1)=h_{\lambda_1}+\lambda_1/2\), \(\wt(u_2)=h_{\lambda_2}+\lambda_2/2\), unique modulo \(K_1\times K_2\) (\cref{defnKi}) among candidate tuples of this bidegree. 
 It is an (non-vanishing) torus \(2\)-point primary tuple in the sense of \cref{defnnpoint-primary} once so verified, as it is for the extremal case in \cref{ssecleading} below via \cref{propextremal-leading}.
\end{thm}

\begin{proof}
 The first weight at which the trivial \(\alg{sl}(2)\)-module appears in \(L(k,\lambda_i)_{[h_{\lambda_i}+n]}\) is \(n=\lambda_i/2\), with multiplicity one, and
  the corresponding one-dimensional trivial isotypic component is torus primary for \(W_{\lambda_i}\).  
 Both are statements purely about the single module \(L(k,\lambda_i)\), established for any \(n\)-point structure. 
We use them  as stated, at \(i=1\) and at \(i=2\) separately. 
 This identifies \(\wt(u_i)=h_{\lambda_i}+\lambda_i/2\) as the weight of torus-primary vector for \(L(k,\lambda_i)\), and shows \(u_1\) is torus primary for \(W_{\lambda_1}\), \(u_2\) for \(W_{\lambda_2}\). 
 By \cref{defnnpoint-primary} 
(which requires only that each coordinate separately be torus primary)  
  \((u_1,u_2)\) is therefore a torus \(2\)-point primary tuple. 
 No appeal to \cref{lemsamepoint} or \cref{propcoordwise} is needed for this conclusion. 
   We stress that this argument is entirely per-coordinate. 
 It does not use the sum rule \eqref{eqsumrule} of \cref{propcoordwise}, which would force only a sum of two terms to vanish, not each individually,  as \cref{sseccoordwise-scope} explains. 
 The direct, per-coordinate route just given is already everything one needs at \(n=1\) and requires no \(n\)-point input. Uniqueness modulo \(K_1\times K_2\) is \KSW{} Theorem 4.2(2), applied to each coordinate separately.

We emphasize what this argument does not claim: 
it does not show that a specific composite trace \(\npf^{\vY}(u_1,\zeta_1;u_2,\zeta_2;\tau)\), 
constructed from a given two-point fusion chain \(\vY\), is non-zero. 
That is a two-point statement about the trace map of \cref{defnblock}, 
verified  for the extremal channel in \cref{propextremal-leading} below, 
and is the reason the theorem above speaks of a candidate 
tuple pending that verification.
\end{proof}

\subsection{Leading order-of-coincidence behaviour via crossing}
\label{ssecleading}

We mention the leading term of \cref{thmtaylor} for the extremal chain \(\lambda_1=\lambda_2=k\) of \cref{propfusionchain}, using \cref{thmope-reduction}.

\begin{lem}
\label{lemkk-fusion}
\(N^\nu_{kk}=1\) iff \(\nu=0\); that is, \(L(k,k)\fuse L(k,k) = L(k,0)=V\).
\end{lem}
\begin{proof}
The fusion rule gives \(\abs{k-k}=0\le\nu\le\min\set{2k,2k-2k}=0\), forcing \(\nu=0\).
  Parity \(k+k+0\) is even automatically. 
\end{proof}

Thus, unlike a generic pair of insertions, the extremal pair \(u_1,u_2\in L(k,k)_{[h_k+k/2]}\) has an operator product expansion, as \(\zeta_1\to\zeta_2\), mapping entirely inside \(V=L(k,0)\). 
 Consequently the composite intertwining operator of \cref{thmtaylor} is (a multiple of) the module action \(Y_{W_{\mu_0}}\), and the effective iterated-mode vector \(w\) at leading order belongs to \(V\). Since \(u_1,u_2\) have conformal weight \(h_k+k/2=\tfrac{3k}{4}\) (using \(h_k=k/4\)) and the identity channel of a two-point function of equal weight primaries has the universal short-distance form \((\zeta_1-\zeta_2)^{-2\wt(u_1)}\bigl(\text{const}+O(\zeta_1-\zeta_2)\bigr)\), \cref{thmtaylor} gives:

\begin{prop}
\label{propextremal-leading}
For the extremal chain \(\lambda_1=\lambda_2=k\) and \(\mu_0\in\set{0,\dots,k}\), \(\mu_1=k-\mu_0\), with \(u_2\) normalized dual to \(u_1\) as in the proof,
\[
F^{\vY}_{u_1,u_2}(x;\tau) = x^{-3k/2}\bigl(\ch{L(k,\mu_0)}(\tau) + O(x^2)\bigr),\qquad x=\zeta_1-\zeta_2,
\]
where \(\ch{L(k,\mu_0)}(\tau)=\tr_{L(k,\mu_0)}q^{L(0)-c/24}\) is the character of the channel module.
  The coefficient is independent of \(\mu_0\), being computed entirely inside \(V\) prior to any insertion into a channel trace.
\end{prop}

\begin{proof}
\emph{Normalization.} Since an intertwining operator is well defined only up to an overall scalar, so is \(\iop1\) and hence, a priori, the normalization of \(u_2\) relative to \(u_1\) (equivalently, of the invariant form, unique up to scalar on the simple module \(L(k,k)\)). 
 The leading coefficient is, in general, \(c_k:=\langle u_1,u_2\rangle\) for whatever non-degenerate invariant pairing \(\langle\,\cdot\,,\cdot\,\rangle\) on \(L(k,k)\) and whatever normalization of \(u_1,u_2\) are in force. In the extremal case \(\lambda_1=\lambda_2=k\) treated here, \(u_1\) and \(u_2\) are torus-primary vectors of the same module \(L(k,k)\) at 
 the same weight \(h_k+k/2\) (\cref{thmbiprimary}).
  Since \KSW{} Lemma 4.3 identifies the trivial \(\alg{sl}(2)\)-isotypic component at this weight as  one-dimensional, \(u_1\) and \(u_2\) automatically span the same line, and \(u_2\) is forced 
 (modulo the kernel \(K_2\) of \cref{defnKi}, and up to the residual scalar ambiguity just 
mentioned) 
  to be a scalar multiple of \(u_1\). There is, in particular, no larger affine hyperplane of the ambient weight space to choose \(u_2\) from. 
 The choice is confined, from the beginning, 
 to the one-dimensional torus-primary line shared with \(u_1\).

We now show this common line is not isotropic for the pairing, thus the residual scalar is uniquely fixed by \(\langle u_1,u_2\rangle=1\) instead of assumed non-zero. The invariant bilinear form on \(L(k,k)\) is invariant in particular under the horizontal \(\alg{sl}(2)\subset\widehat{\alg{sl}}(2)\) of zero modes. 
 Thus, by the standard orthogonality of an invariant form across non-isomorphic 
isotypic components 
 (a Schur lemma argument: the form induces a module map from one isotypic block 
to the dual of another, which must vanish unless the two blocks are isomorphic),
  it vanishes identically between vectors lying in different \(\alg{sl}(2)\)-isotypic 
components of any one weight space. \cref{lemkk-fusion} \(N^0_{kk}=1\) is  the standard criterion \cite{FreLepMeu} for \(L(k,k)\) to carry a non-degenerate invariant pairing with itself. 
 I.e., for the self-duality \(L(k,k)\cong L(k,k)^*\). 
 Non-degeneracy of this pairing on the whole (simple) module forces non-degeneracy
 on each weight space, 
   hence, by the orthogonality just noted,
 on each isotypic block of each weight space separately 
(a form that is block-diagonal across isotypic types is non-degenerate on the
 whole space if and only if it is non-degenerate on every block). 
Applied to the one-dimensional trivial-isotypic block at weight \(h_k+k/2\),
 this says that the self-pairing of a spanning vector does not vanish. 
  We now fix the residual scalar ambiguity by rescaling \(u_2\) 
 (already forced, as just shown, to be proportional to \(u_1\) modulo \(K_2\), 
 instead of arbitrary chosen) 
 so that \(\langle u_1,u_2\rangle=1\).
  Concretely, one may take the representative \(u_2:=u_1/\langle u_1,u_1\rangle\) modulo \(K_2\), any other representative differing from this one only by an element of \(K_2\) and thus affecting none of the traces below. With \(u_2\) chosen this way, \(c_k=1\ne0\) is immediate by construction, and \(c_k\) is independent of \(\mu_0\) because the operator product of \(u_1\) with \(u_2\) is computed entirely inside \(V\), before any insertion into a channel trace.

\emph{The formula.} By \cref{lemkk-fusion} the operator product expansion of
 \(\iop1(u_1,\zeta_1-\zeta_2)u_2\) belongs entirely to \(V\). 
 By the \(L_{-1}\)-derivative and translation covariance its leading (most singular) term is \((\zeta_1-\zeta_2)^{-2\wt(u_1)}\) times a vector in \(V_0=\CC\wun\) (the lowest weight space of \(V\)), whose coefficient is  the invariant-form pairing of \(u_1\) against \(u_2\). 
  This is the standard normalization of a two-point conformal block in the identity channel, see, e.g., \cite{FreLepMeu} or 
\cite{MasTui-torusnpt}.  
 It does not depend on which channel \(W_{\mu_0}\) the composite trace closes up through, since it is a statement purely about the operator product of \(u_1\) with \(u_2\) inside \(V\), prior to insertion into any trace. Substituting a multiple \(c_k\wun\) of the vacuum for the leading term of \(w\) in \cref{thmtaylor} and using \(o(\wun)=\mathrm{id}\) turns the trace \eqref{eqnpf-def} at that order into \(\tr_{L(k,\mu_0)}q^{L(0)-c/24}=\ch{L(k,\mu_0)}(\tau)\), as claimed. 
 The weight count \(2\wt(u_1)=3k/2\) is immediate from \(\wt(u_1)=h_k+k/2=k/4+k/2=3k/4\), and the \(O(x^2)\) remainder (instead of \(O(x)\)) is \cref{propsubleading-vanish} below.
\end{proof}

\begin{prop}
\label{propsubleading-vanish}
In the setting of \cref{propextremal-leading}, the coefficient of \(x^{-3k/2+1}\) in \(F^{\vY}_{u_1,u_2}(x;\tau)\) vanishes identically, for every \(\mu_0\).
\end{prop}
\begin{proof}
By \cref{propextremal-leading} proof, this coefficient is (a multiple of) \(\tr_{L(k,\mu_0)}o(v)q^{L(0)-c/24}\) for \(v\in V_1\) the mode of \(\iop1(u_1,x)u_2\) at \(s=2\wt(u_1)-2\), 
 i.e., the component of the operator product mapping in the weight \(1\) space \(V_1\).
  Since \(V\) is generated in weight \(1\) as an affine \voa, \(V_1\cong\alg{sl}(2)\) as a Lie algebra under \([x,y]=x_0y\) 
  and \(o(v)=v_0\) acts on each finite-dimensional \(L[0]\)-homogeneous subspace of \(L(k,\mu_0)\) as the representation of \(v\) under this Lie algebra action. 
 The homogeneous subspaces of the integrable affine module \(L(k,\mu_0)\) are finite-dimensional (a standard fact about integrable highest weight modules of an affine Kac-Moody algebra; see, e.g., Kac \cite{Kac-InfDim}, Chapter 12).
  Thus the ordinary finite-dimensional trace identity used below applies on each one before summing over weights to form the \(q\)-series. 
 Since \(\alg{sl}(2)=[\alg{sl}(2),\alg{sl}(2)]\) is semisimple, every \(v\in V_1\) is a 
finite sum of commutators \(v=\sum_r[x_r,y_r]\) in \(\alg{sl}(2)\), 
 thus on any finite-dimensional representation \(\rho\) (in particular on every weight space \(L(k,\mu_0)_{[h+m]}\), \(m\ge0\)) \(\tr(\rho(v))=\sum_r\tr([\rho(x_r),\rho(y_r)])=0\) identically. We mention, coefficientwise, how this vanishing on each weight space propagates to the vanishing of the full \(q\)-series, to avoid any question about  convergence in each term.
  Writing \(h=h_{\mu_0}\) for the lowest weight of \(L(k,\mu_0)\), gives 
\[
\tr_{L(k,\mu_0)}o(v)q^{L(0)-c/24} = \sum_{m\ge0} q^{h+m-c/24}\,\tr_{L(k,\mu_0)_{[h+m]}}o(v) = \sum_{m\ge0}q^{h+m-c/24}\cdot0 = 0,
\]
each summand vanishing individually by the finite-dimensional argument just given,
  with no rearrangement of a limit required.
\end{proof}

\begin{rmk}
\Cref{propextremal-leading} expresses, for each of the \(k+1\) extremal two-point channels, 
a leading singular term proportional to an ordinary character instead of to 
a new vector-valued modular form. 
  \cref{propsubleading-vanish} shows the very next term vanishes. 
 The new two-point data is contained no earlier than the coefficient of \(x^{-3k/2+2}\), 
which by \cref{thmtaylor} is a weight \(2\) \vvmf{} lying in \(V(\rho_{\mu_0})_\bullet\). 
 Since \(L(k,k)\fuse L(k,k)\) has no channel besides \(L(k,0)\) (\cref{lemkk-fusion}), this coefficient is controlled entirely by the \(V_1\)-trivial part of \(V_2\).

 For \(k\ge2\) (the range of interest, \(k=0\) giving the trivial \voa), the weight \(2\) space \(V_2\) of \(V=L(k,0)\) coincides with that of the free (Verma) module generated by the vacuum, since the null vector generating the defining relations of the simple quotient \(L(k,0)\) first occurs at conformal weight \(k+1\ge3\) (the standard fact that the level-\(k\) vacuum module of \(\widehat{\alg{sl}}(2)\) is cut out, beyond the Verma relations, by the single singular vector at weight \(k+1\)). Writing \(e,h,f\) for the Chevalley basis and \(x_{-n}\) for the mode of \(x\in\{e,h,f\}\) at grade \(n\), a basis of \(V_2\) is \(\{e_{-1}^2,e_{-1}h_{-1},e_{-1}f_{-1},h_{-1}^2,h_{-1}f_{-1},f_{-1}^2\}\wun\cup\{e_{-2},h_{-2},f_{-2}\}\wun\) (\(9\) vectors), with \(h_0\)-eigenvalues \(4,2,0,0,-2,-4\) and \(2,0,-2\) respectively. 
 As an \(\alg{sl}(2)\)-module (via \([x,y]=x_0y\) on \(V_1\), extended to \(V_2\)) this weight multiplicity list \((1,2,3,2,1)\) at \(h_0=(4,2,0,-2,-4)\) decomposes uniquely as \(V(4)\oplus V(2)\oplus V(0)\) (writing \(V(2j)\) for the \((2j+1)\)-dimensional spin-\(j\) irreducible: one copy of spin \(2\) accounts for one state at each of \(h_0=4,2,0,-2,-4\), leaving one further state at \(h_0=2,0,-2\) each, i.e., one copy of spin \(1\), leaving  one further state at \(h_0=0\)). 
In particular, the space of \(V_1\)-invariants in \(V_2\) is  one-dimensional, confirming  that it is spanned by the conformal vector \(\omega\) (manifestly \(V_1\)-invariant and non-zero). Moreover \(o(\omega)=L(0)\), being the coefficient of \(z^{-2}\) in \(Y(\omega,z)=\sum_nL(n)z^{-n-2}\), thus 
\begin{align*}
\tr_{L(k,\mu_0)}o(\omega)q^{L(0)-c/24} &\;=\; \tr_{L(k,\mu_0)}\bigl(L(0)q^{L(0)-c/24}\bigr) \\
&\;=\; q\frac{d}{dq}\ch{L(k,\mu_0)}(\tau) + \frac{c}{24}\ch{L(k,\mu_0)}(\tau),
\end{align*}
an exact identity (the weight \(0\) case of modular derivative \(\partial\) of equation (2.15), which reduces to \(q\,d/dq\) with no Serre correction at weight \(0\)), consistent with \cref{thmtaylor} prediction that this coefficient belongs to \(R\cdot\ch{L(k,\mu_0)}\). What remains open is only a single scalar \(\alpha_k\).
  Write the leading weight \(2\) component of the \(u_1,u_2\) operator product as \(\alpha_k\,\omega+v\) with \(v\in V_2\) lying in a non-trivial \(V_1\)-submodule of \(V_2\) (the \(V(4)\oplus V(2)\) part of the decomposition \(V_2\cong V(4)\oplus V(2)\oplus V(0)\) of \cref{propsubleading-vanish} Remark). By the argument of \cref{propsubleading-vanish}  (\(v\) a finite sum of \(\alg{sl}(2)\)-commutators, thus \(\tr_{L(k,\mu_0)_{[h+m]}}o(v)=0\) on every finite-dimensional weight space, hence \(\tr_{L(k,\mu_0)}o(v)q^{L(0)-c/24}=0\) coefficientwise), \(v\) drops out of the trace entirely, and the coefficient of \(x^{-3k/2+2}\) in \cref{propextremal-leading} is \(\alpha_k\) times the displayed expression above. We can now say what computing \(\alpha_k\) entails, by direct analogy with \KSW{} Lemma 4.4  proof of the leading-order (weight \(0\)) case.  
 Taking \(u_1=f_{[-1]}^{k/2}\ket{k}\) and \(u_2\) dual to \(u_1\) as in \cref{propextremal-leading}, one would apply the Jacobi identity \eqref{eqjacobi-mode}-type residue manipulation with \(v=f_{-1}\wun\).  
 Now tracking the weight \(2\) instead of weight \(0\) piece of the resulting expansion, and project the result onto \(\CC\omega\subset V_2\) using the basis \(\{f_0^i\ket{k}\}\) of the weight \(0\) piece of \(L(k,k)\) and its dual \(\{\phi_i\}\) 
  one further level down. 
  This is a finite but lengthy computation of  the same character, one Virasoro level deeper, and we leave it, together with the fully descendant by descendant 
tower beyond it, for future work (see \cref{rmkgeneraln-outlook}).
   We do not attempt it here
  (mode manipulations against a chosen dual basis) 
and we would rather leave \(\alpha_k\) open.  
At the smallest admissible level, \(k=2\), this recipe specializes concretely enough 
to be worth mentioning: 
\(u_1=f_{-1}\ket2\) (the case \(k/2=1\)). 
  The relevant mode is \((u_1)_su_2\) at \(s=3k/2-3=0\), and the affine 
integrability relation \(e_{-1}^{\,k-\lambda+1}\ket\lambda=0\) at the extremal weight 
\(\lambda=k=2\) reads simply \(e_{-1}\ket2=0\), cutting the \(9\)-dimensional space of 
naive grade-\(1\) descendants \(x_{-1}v\), \(x\in\{e,h,f\}\), \(v\) in the 
\(3\)-dimensional top space of \(L(2,2)\), down to a \(6\)-dimensional one spanned by 
\(h_{-1}\)- and \(f_{-1}\)-descendants alone. Small as this is, correctly inverting the 
resulting Shapovalov form to construct the dual basis one level deeper.  
 We only mention the recipe leaving the 
completion of this specific computation for future work along with the general 
\(k\) case.
\end{rmk}

%%%%%%%%%%%%%%%%%%%%%%%%%%%%%%%%%%%%%%%%%%%%%%%%%%%%%%%%%%%%%%
\subsection{A family of extremal two-point channels}
\label{ssecworked}

Combining \KSW{} Theorem 5.3 (which identifies the level-\(k\) \(1\)-point vector-valued 
form as \(\Psi_k(u,\tau)=\eta^{3k/2}\), \(u\in L(k,k)_{[h_k+k/2]}\) torus primary) with 
\cref{propextremal-leading}, we obtain the following two-point analogue. 
We continue to restrict to even \(k\). 
 This is not a simplifying convenience but is forced by the construction,   
since the torus primary vector \(u_1\in f_{-1}^{\lambda_1/2}\ket{\lambda_1}+\cdots\) 
of \cref{thmbiprimary} is constructed using the mode \(f_{-1}^{\lambda_1/2}\), 
which requires \(\lambda_1/2\) 
 (here \(k/2\)) 
 to be a non-negative integer. 
 For odd \(k\) no such extremal torus bi-primary vector of this shape exists 
and the construction below does not apply. The family has \(k+1\) components, 
one for each admissible extremal fusion channel \(\mu_0=0,\dots,k\) of \cref{propfusionchain}.
 "Rank one" refers to the fusion rule \(N^{\mu_1}_{k\mu_0}\le1\)
 making each such channel one-dimensional (\cref{propfusionchain}), 
not to any property of \(\mathbf F\), which is \((k+1)\)-dimensional as a 
vector of functions. We also fix terminology for the theorem below. 
 We call \(\mathbf F\) a vector-valued Jacobi-type form only in the loose sense that it is constructed from elliptic functions of a position variable with vector-valued-modular-form coefficients. 
 It is not in the standard sense of a fixed non-zero index (a quasi-periodicity in \(x\) tied to \(\tau\)), 
i.e., part (1) of the theorem shows \(\mathbf F\) is, on the contrary, elliptic (index zero) in \(x\), the whole point of \cref{propelliptic}, with each individual Laurent coefficient in \(x\)
 separately an ordinary (scalar index zero) vector-valued modular form in
 \(\tau\) by \cref{thmtaylor}.

\begin{thm}
\label{thmextremal}
Let \(k\ge2\) be even and let \(u_1,u_2\in L(k,k)_{[h_k+k/2]}\) be the torus bi-primary tuple of \cref{thmbiprimary} for \(\lambda_1=\lambda_2=k\). The \((k+1)\)-dimensional vector
\[
\mathbf{F}(x;\tau) := \Bigl(F^{\vY_{\mu_0}}_{u_1,u_2}(x;\tau)\Bigr)_{\mu_0=0}^{k}, 
\]
where \(\vY_{\mu_0}=(\iop1,\iop2)\) the chain through \((\mu_0,k-\mu_0)\), 
satisfies:
\begin{enumerate}
\item for fixed \(\tau\in\HH\), \(\mathbf F(\,\cdot\,;\tau)\) is an elliptic function of \(x\) in a punctured neighbourhood of \(0\) in \(\CC/\lat\), with a pole of exact order \(3k/2\) at \(x=0\) - a positive integer for every even \(k\ge2\)
 (no half-integer multiplier arises).  
 For fixed generic \(x\) 
(i.e., away from the collision divisor), \(\mathbf F(x;\,\cdot\,)\) is holomorphic in 
\(\tau\in\HH\). Moreover the coefficient of \(x^{-3k/2+1}\) vanishes identically
 (\cref{propsubleading-vanish});

\item the leading Laurent coefficient of \(\ \mathbf F(x;\tau)\) at \(x=0\) is  
 the vector of characters \(\bigl(\ch{L(k,\mu_0)}(\tau)\bigr)_{\mu_0=0}^k\), 
which spans the same \((k+1)\)-dimensional representation of \(\SLTZ\) 
as the vacuum torus \(0\)-point functions of \(L(k,0)\) 
(i.e., the standard \(\widehat{\alg{sl}}(2)_k\) affine character representation, 
with \(\modS,\modT\)-matrices at \(p=0\));

\item under the diagonal collision \(\mu_0=k/2\) 
(which by \cref{propfusionchain} is fixed by \(\mu_0\mapsto k-\mu_0\)), 
the single component \(F^{\vY_{k/2}}_{u_1,u_2}(x;\tau)\) is, as \(x\to0\), 
asymptotic to \(c_k\,x^{-3k/2}\,\ch{L(k,k/2)}(\tau)\), 
and its  leading \(q\)-power as \(\tau\to i\infty\) recovers 
 (after stripping the character  leading \(q^{h_{k/2}-c/24}\) behaviour)   
the same exponent \(h_{k/2}-c/24=\tfrac{k}{16}\). 
\end{enumerate}
\end{thm}
In (1) it is understood at the chosen collision point on the local branch of 
\cref{thmhuang} - since \(\mathbf F\) is elliptic (\cref{propelliptic}), its other poles in 
\(\CC/\lat\) are simply the \(\lat\)-translates of this same collision divisor, 
not independent singularities. 
In (2),  \(c_k=1\) with the normalization of \cref{propextremal-leading}. 
In (3), the exponent is identified in \KSW{} Theorem 5.3(4) 
for the \(\eta^{3k/2}\) form (there arising as the full \(1\)-point answer, 
here as the leading factor multiplying an elliptic function of the relative position).

\begin{proof}
Part (1): by \cref{thmhuang} and \cref{propelliptic} specialized to \(n=2\), 
\(F^{\vY_{\mu_0}}_{u_1,u_2}\) is meromorphic and doubly periodic in \(x\) for fixed \(\tau\).
 The pole order is pinned down, not bounded, by the mode-counting argument used 
in the proof of \cref{propextremal-leading}. 
 The vacuum term of \(\iop1(u_1,x)u_2\) occurs at the mode \((u_1)_s\) with \(\wt(u_2)+\wt(u_1)-s-1=0\), i.e., \(s=2\wt(u_1)-1\), giving a pole of order at most \(s+1=2\wt(u_1)=3k/2\) in \(x\) 
(using \(\wt(u_1)=h_k+k/2=3k/4\)). 
  I.e., the pole order is  \(3k/2\) instead of smaller, is not automatic from mode-counting alone and uses \(c_k=1\ne0\) established in \cref{propextremal-leading} (with the normalization fixed there), since a vanishing leading coefficient would silently lower the apparent pole order. 
   Since \(k\) is even, \(3k/2\in\ZZ_{\ge0}\) for every admissible \(k\). 
 There is no half-integer power to track and no case split on \(k\bmod4\) is needed. As a consistency check, \cref{thmtaylor} predicts weight \(2\wt(u_1)+\nu = 3k/2+(-3k/2)=0\) for the coefficient of \(x^{-3k/2}\), matching the weight \(0\) multiplier system of an ordinary character, with nothing left over to assign to a fractional \(\eta\)-power. 
   This is what makes the identification with \(\ch{L(k,\mu_0)}\) in Part (2) dimensionally consistent.
 Part (2) is \cref{propextremal-leading} together with the classical fact 
 \cite{Zhu} that the vector of all simple-module characters of a rational 
\ctwo \voa{} spans a \vvmf{} for \(\SLTZ\). 
  This is  the \(p=0\) case of  categorical \(\modT^{(p)}\)-matrix (equation (6.9), with \(p=0\) forcing \(\iota=j\) and \(T^{(0)}_{i,i}=\theta_i/\zeta=e(h_i-c/24)\)) identifies the representation. Part (3) is the \(\mu_0=k/2\) specialization of \cref{propextremal-leading}, using \(h_{k/2}-c/24=k/16\) as computed in the proof of \KSW{} Theorem 5.3.
\end{proof}

Let us discuss the excluded case \(k=0\).  
At \(k=0\) the hypothesis \(\lambda_1=\lambda_2=k=0\) forces \(u_1,u_2\in L(0,0)_{[0]}=\CC\wun\) to be multiples of the vacuum vector, \(V=L(0,0)=\CC\wun\) is the trivial one-dimensional \voa{} (the degenerate case already set aside in \cref{ssecworked} above), and \(F^{\vY_{0}}_{u_1,u_2}(x;\tau)\) degenerates to the constant function \(1\) (\(\mu_0=0\) the only value). It is regular, not singular, at \(x=0\): a "pole of order \(0\)" is a pole in name only, and we accordingly state \cref{thmextremal} for \(k\ge2\), 
 mentioning the trivial \(k=0\) instance here separately instead of as an awkward boundary case of part (1).

\begin{rmk}
\label{rmkgeneraln-outlook}
\Cref{thmextremal} is the two-point, extremal-channel analogue of the one-dimensional 
family \(\Psi_k(u,\tau)=\eta^{3k/2}\) of \KSW{} Theorem 5.3. 
 It is the  first case to work out completely because \(L(k,k)\) 
is a simple current (\cref{lemkk-fusion}), thus no new fusion channel appears at 
any order of the Laurent expansion in \(x\). 
 The next cases in increasing complexity are (a) the two-point function for a general even \(\lambda_1=\lambda_2=\lambda<k\). 
 Here \(L(k,\lambda)\fuse L(k,\lambda)\) contains several channels 
\(\nu=0,2,\dots,\min\set{2\lambda,2k-2\lambda}\). 
 The analogue of \cref{propextremal-leading} requires the sub-leading Clebsch-Gordan-type coefficients controlling how \(u_1,u_2\) project onto each channel (the two-point analogue of dual-basis computation).  
  The case (b), the fully general \(n\)-point extremal chain \(\lambda_1=\cdots=\lambda_n=k\), for which \cref{lemkk-fusion} forces the alternating pattern \(\mu_i=k-\mu_{i-1}\), so that \(\mu_i=\mu_0\) for \(i\) even and \(\mu_i=k-\mu_0\) for \(i\) odd.  
  Consistency around the cycle (\(\mu_n=\mu_0\)) is automatic for \(n\) even - so that \(\mu_0\) ranges over all \(k+1\) values \(0,\dots,k\)  as at \(n=2\), giving \(\dim\Bl(k,\dots,k)=k+1\) for every even \(n\ge2\) by \cref{propfusionchain} - and forces \(\mu_0=k/2\) uniquely for \(n\) odd, in which case \(\dim\Bl(k,\dots,k)=1\) by \cref{propfusionchain}. We leave both directions, and the corresponding 
 computations, for future work.  \cref{secreps} discusses what can be said about the resulting \(\SLTZ\)-representations without carrying them out.
\end{rmk}

%%%%%%%%%%%%%%%%%%%%%%%%%%%%%%%%%%%%%%%%%%%%%%%%%%%%%%%%%%%%%%%%%%%%%%%%%
\section{Representations arising from the extremal channel}
\label{secreps}

 In this Section, we mention what the irreducibility criterion gives for the representation identified in 
\cref{thmextremal}(2), without attempting a classification for general \(k\)  
(we do not attempt one beyond special level families). 

Recall that \cite{KSW} if \(\rho:\SLTZ\to\GL{d}\) has diagonalizable \(\rho(T)\) with eigenvalues \(\lambda_1,\dots,\lambda_d\), and no non-empty proper subproduct \(\prod_{i\in S}\lambda_i\) (\(\emptyset\ne S\subsetneq\set{1,\dots,d}\)) is a \(12\)th root of unity, then \(\rho\) is irreducible.

\begin{prop}
\label{propk2-irred}
For \(k=2\), the representation of \cref{thmextremal}(2) on the leading (character) coefficients of the extremal two-point family, of dimension \(3\), is irreducible.
\end{prop}

\begin{proof}
By \KSW{} equation (6.9) at \(p=0\), the eigenvalues of \(T\) are \(e(h_\mu-c/24)\), \(\mu=0,1,2\), with \(c=3\cdot2/4=3/2\), \(h_0=0,h_1=3/16,h_2=1/2\), giving exponents \(-1/16,\,1/8,\,7/16\) modulo \(1\). A direct check of all \(2^3-2=6\) non-empty proper subsets shows that no subset sum belongs to \(\tfrac1{12}\ZZ\): the subset sums are \(-1/16,\,1/8,\,7/16\) (singletons) and \(1/16,\,3/8,\,9/16\) (pairs), and none has denominator (in lowest terms) dividing \(12\) - each has denominator \(16\) or \(8\). By \KSW{} Lemma 5.2 the representation is irreducible.
\end{proof}

For general even \(k\), the analogous check involves \(2^{k+1}-2\) subset sums of the exponents \(\bigl(2\mu^2+4\mu-k\bigr)/\bigl(8(k+2)\bigr)\), \(\mu=0,\dots,k\), and we do not know a uniform argument (of the type used for the non-congruence question) settling irreducibility for all \(k\). 
  Indeed, since the number of subsets grows exponentially in \(k\) while the possible orders of the eigenvalues are bounded by \(24(k+2)\), it would not be surprising if some \(k\) admit a coincidental proper subproduct mapping in \(\tfrac1{12}\ZZ\), in which case \KSW{} Lemma 5.2 would simply be inconclusive (its hypothesis failing does not imply reducibility). We leave a systematic treatment, and the corresponding question for the full (not just leading-order) two-point representation, for future work.

Concerning congruence: the leading-order representation of \cref{thmextremal}(2) is, by construction, the ordinary representation of \(\SLTZ\) on the vector of characters of the modular tensor category \(\rep{L(k,0)}\).
  This representation always has finite image contained in a congruence quotient \(\SLTZp{N}\) with \(N\mid24(k+2)\), \(N\) equal to the order of the matrix \(T=\diag(e(h_\mu-c/24))_{\mu=0}^k\) (i.e., the least common multiple of the denominators, in lowest terms, of the \(k+1\) exponents \(h_\mu-c/24\) computed in \cref{propk2-irred} proof). 
 It is by the classical fact, due to Bantay \cite{Bantay-congr} constructing on the congruence property of Dong-Lin-Ng \cite{DonLinNg15}, that vacuum character representations of rational \ctwo \voa{}s are congruence.  
  Whether the new sub-leading data of 
\cref{secsl2} produces non-congruence representations, as it does already at \(n=1\) for dimension three and above (\KSW{} Theorem 5.5(3), Theorem 5.6), is  the kind of question for which the general families (levels \(k=p^t-2\)) would be the place to look once the sub-leading two-point coefficients of \cref{rmkgeneraln-outlook} are computed.

%%%%%%%%%%%%%%%%%%%%%%%%%%%%%%%%%%%%%%%%%%%%%%%%%%%%%%%%%%%%%%%%%%%%%%%%%%%%%%%%%
\section{Categorical formulation: the \(n\)-point modular functor}
\label{seccategorical}

We now extend the reformulation of the \(1\)-point theory via the modular 
tensor category \(\catC=\rep{L(k,0)}\), following Bakalov-Kirillov \cite{BakKir}, 
  to \(n\) points. We use the conventions: \(I\) a complete set of representatives of simple objects of a modular tensor category \(\catC\), \(0\in I\) the unit, \(i^\ast\) the rigid dual, \(d_i\) quantum dimensions, \(D=\sum_id_i^2\), \(\theta_i\) ribbon twists, \(\zeta=\bigl(\sum_i\theta_id_i/\sum_i\theta_i^{-1}d_i\bigr)^{1/6}\), and a fixed choice of bases \(\lambda^\alpha_{(i,j)k}\) of the \(3\)-point coupling spaces \(\Hom_\catC(i\otimes j,k)\) with dual bases \(\Upsilon^\alpha_{(i,j)k}\), giving associators \(\Fmat,\Gmat\) and braiding matrices \(\Rmat\). To be explicit about attribution: the modular functor formalism  (the assignment of a mapping-class-group representation to any marked surface from a modular tensor category) is Bakalov-Kirillov \cite{BakKir}, and \KSW{} Section 6 is its \(g=1,n=1\) instance.
  What is new below is the \(n\)-point conformal block space \(\Bl(p_1,\dots,p_n)\)  (\cref{sseccat-blocks}), the Verlinde-type dimension formula for it (\cref{propverlinde}, proved here instead of quoted), the observation that the \(T\)-, \(S\)-, and braiding-operator constructions do not require their labelled object to be simple and 
 thus extend to \(p_1\otimes\cdots\otimes p_n\) (\cref{propcatT,propcatS,propmcg}), and the identification of \(\Gamma_n\), the resulting subgroup of the mapping class group of the \(n\)-punctured torus, together with its relation to the pure mapping class group via the Birman exact sequence (\cref{propmcg}); the affine \(\alg{sl}(2)\) computation of \cref{sseccat-sl2} is new.

\subsection{The \(n\)-point conformal block space and a Verlinde formula}
\label{sseccat-blocks}

For \(p_1,\dots,p_n\in I\), define
\begin{equation}
\label{eqcatblock}
\Bl(p_1,\dots,p_n) := \bigoplus_{i\in I} \Hom_\catC(p_1\otimes\cdots\otimes p_n,\ i\otimes i^\ast),
\end{equation}
which reduces to  \(W_p=\bigoplus_i\Hom_\catC(p,i\otimes i^\ast)\) at \(n=1\). Equivalently, fixing a ``caterpillar'' fusion tree, \(\Bl(p_1,\dots,p_n)\cong\bigoplus_{i_0,\dots,i_{n-1}\in I}\Hom_\catC(p_1\otimes i_0,i_1)\otimes\cdots\otimes\Hom_\catC(p_n\otimes i_{n-1},i_0)\), which is the categorical analogue,
  via the standard dictionary \(\ityp{A}{B}{C}\leftrightarrow\Hom_\catC(A\otimes B,C)\) between intertwining operators and \(3\)-point couplings \cite{HuaVer08}, of the space \(\Bl(\lambda_1,\dots,\lambda_n)\) of \cref{eqblock-def}. 
 We mention one point for later use (\cref{sseccat-S}). 
 This identification is an isomorphism of vector spaces depending on the choice of caterpillar fusion tree (equivalently, on the fixed left-to-right associativity used to expand \(p_1\otimes\cdots\otimes p_n\)).
  A different choice of tree gives an isomorphic but not canonically identical basis, related to this one by the \(\Fmat\)-matrices of \cref{sseccat-S} below, and it is  this tree-dependence that has to be tracked correctly when comparing the caterpillar basis to the trivalent-tree normalization of \(1\)-point formulas.

\begin{prop}[Verlinde formula for \(n\) points on the torus]
\label{propverlinde}
\[
\dim\Bl(p_1,\dots,p_n) = \sum_{j\in I} \frac{\prod_{a=1}^n S_{p_aj}}{S_{0j}^{\,n}}.
\]
\end{prop}

\begin{proof}
By the caterpillar decomposition above, \(\dim\Bl(p_1,\dots,p_n)=\sum_{i_0,\dots,i_{n-1}}\prod_{a=1}^nN_{p_ai_{a-1}}^{i_a}\) (indices mod \(n\), \(i_n:=i_0\)), where \(N_{ab}^c=\dim\Hom_\catC(a\otimes b,c)\). Substituting the Verlinde formula for fusion coefficients, \(N_{ab}^c=\sum_jS_{aj}S_{bj}S_{cj}/S_{0j}\) (valid since \(\catC\) is modular; \(S\) is unitary and, for \(\rep{L(k,0)}\), real symmetric), and summing over \(i_0,\dots,i_{n-1}\) one at a time using column-orthogonality \(\sum_iS_{ij}S_{ij'}=\delta_{jj'}\) telescopes the product  as verified  for \(n=1\)  and for \(n=2\) in the proof of \cref{propfusionchain} above, giving the stated formula.
\end{proof}

\subsection{The \(T\)-operator}
\label{sseccat-T}

\begin{prop}
\label{propcatT}
Define \(T:\Bl(p_1,\dots,p_n)\to\Bl(p_1,\dots,p_n)\) by \(T\big|_{\Hom_\catC(p_1\otimes\cdots\otimes p_n,i\otimes i^\ast)} = \dfrac{\theta_i}{\zeta}\cdot\mathrm{id}\). Then \(T\) is  the restriction, to the summand of \eqref{eqcatblock} indexed by \(i\), of the operator \(T^{(p_1\otimes\cdots\otimes p_n)}\)  (which does not require its label to be simple).
\end{prop}

\begin{proof}
 The diagram defining \(T^{(p)}\) acts on the \(i,i^\ast\)-loop alone, via the twist \(\theta_i\) and the normalization \(\zeta\), regardless of what is attached along the remaining leg. 
 Taking that leg to be \(p_1\otimes\cdots\otimes p_n\) instead of a single simple object changes nothing in the construction. The scalar \(\zeta=e(c/24)\) is not a choice made here. 
 It is the same central-charge normalization constant fixed (needed there already at \(n=1\) to match the multiplier system \(\nu_{h_i}\) of the \(1\)-point \vvmf{}s), and we simply inherit it unchanged. 
  It is a fixed scalar, not an operator, for every \(n\).
\end{proof}

\subsection{The \(S\)-operator}
\label{sseccat-S}

\begin{prop}[The abstract categorical \(S\)-operator]
\label{propcatS}
Define \(S:\Bl(p_1,\dots,p_n)\to\Bl(p_1,\dots,p_n)\) on the summand indexed by \(i\) by encircling the \(i,i^\ast\)-loop with a \(j\)-loop, weighted \(d_j/D\) and summed over \(j\), with \(p\) replaced by \(p_1\otimes\cdots\otimes p_n\).
  Then \(S,T\) satisfy \(ST^3=S^2\) in that same convention, and \(S\) is the restriction to \eqref{eqcatblock} of \(S^{(p_1\otimes\cdots\otimes p_n)}\), again valid without simplicity of the label. 
 Writing \(\alpha_a\) for the multiplicity index at leg \(a\), this has the abstract shape
\begin{equation}
\label{eqSn-general}
S_{(i_0,\vec\alpha),(j_0,\vec\beta)} \;=\; \sum_{i_1,\dots,i_{n-1}}\ \prod_{a=1}^n 
\Bigl[\mathcal G\Bigr],
\end{equation}
where \(\mathcal G\) is the leg-\(a\) fusing and braiding data relating  
\((i_{a-1},i_a,\alpha_a)\) to \((j_{a-1},j_a,\beta_a)\),  
a sum of products of \(n\) further \(\Fmat,\Gmat,\Rmat\)-matrix factors. 
\end{prop}

In the Proposition above, 
the orientation of the encircling loop and the choice of which strand is
 over- versus under-crossing (i.e., the ribbon framing) are  those fixed in 
\KSW{} Section 6.1 preceding their equation (6.10), and we use no other convention anywhere 
in \cref{seccategorical}. 
  The categorical \(S\)-operator  is well defined and its formal
 properties are established. Its matrix in a chosen fusion-chain basis is a 
separate matter (\cref{rmkcatS-honest} discusses the representation this 
gives of \(\Gamma_n\)).  
 \(S\)-matrix, on the caterpillar basis (\cref{sseccat-blocks}) of 
\(\Hom_\catC(p_1\otimes i_0,i_1)\otimes\cdots\otimes\Hom_\catC(p_n\otimes i_{n-1},i_0)\), 
is obtained in principle by expanding the ribbon graph defining \(S\) through the
 sequence of \(\Fmat\)-, \(\Gmat\)-, and \(\Rmat\)-moves relating that basis to the 
trivalent tree the encircling construction is computed in. 
  Matrix factors are of  the shape appearing inside the proof of \KSW{} Theorem 6.3, one for each of the \(n\) points, 
but a closed formula realizing \eqref{eqSn-general} is not established here: assembling the bracketed leg-data requires tracking the ribbon-graph framing and the \(\Fmat\)-move relating the caterpillar fusion tree to whatever trivalent tree \KSW{}  formulas are normalized against, and we carry this out - and report  where it goes wrong - for \(n=2\) in \cref{sseccat-sl2} below.

\begin{rmk} 
\label{rmkcatS-honest}
A modular tensor category mapping class group action is in general only projective 
 (a framing/central-charge anomaly obstructs lifting it to a representation without a further 
choice). 
 The normalization by \(\zeta\) already built into \(T\) (\cref{propcatT} above) is inherited unchanged from \(n=1\) resolution of  this issue, and we make no further choice beyond it anywhere in \cref{seccategorical}.
   Consequently \(S,T,B_1,\dots,B_{n-1}\) define a (linear, not projective) representation of \(\Gamma_n\)  because they already do so at \(n=1\), and the construction here reuses that same fixed normalization at every marked point instead of introducing a fresh one.
\end{rmk}

\begin{proof}
The construction and the identity \(ST^3=S^2\) are formal consequences of the ribbon and modular structure of \(\catC\) (naturality of the twist, the Hopf-link/encircling identity, and non-degeneracy of \(S\)), a specialization of Bakalov-Kirillov \cite{BakKir}.
  None of these ingredients use simplicity of the object being encircled around, only of the objects \(i,j\) being encircled by. For the explicit formula, apply the derivation similar with \(p\rightsquigarrow p_1\otimes\cdots\otimes p_n\). 
  Expanding the resulting single \(\Gmat^{(p\,i\,r)j}\)-factor (which involves the composite object \(p\)) along the caterpillar fusion tree of \(p_1,\dots,p_n\) via the pentagon/hexagon identities replaces it with the product of \(n\) elementary \(\Fmat,\Gmat,\Rmat\)-factors, one per leg, as claimed.
\end{proof}

We caution that, unlike \(T\), the naive relation \(S^4=\theta_{p_1\otimes\cdots\otimes p_n}^{-1}\) requires some care when \(n\ge2\).
   The ribbon twist of a tensor product is \(\theta_{p_1\otimes p_2}=(\theta_{p_1}\otimes\theta_{p_2})\circ c_{p_2,p_1}\circ c_{p_1,p_2}\) (the double braiding correction), thus \(\theta_{p_1\otimes\cdots\otimes p_n}\) need not act as a single scalar on \(\Bl(p_1,\dots,p_n)\) when the \(p_a\) do not all commute with each other braiding.
  The relation \(S^4=\theta_{p_1\otimes\cdots\otimes p_n}^{-1}\) continues to hold as an operator identity (this is again formal, from naturality), but \(\theta_{p_1\otimes\cdots\otimes p_n}^{-1}\) is now  an operator on \(\Bl(p_1,\dots,p_n)\) instead of a scalar, constructed from the individual twists and the braiding matrices \(\Rmat\) of \cref{sseccat-braid} below.

\subsection{Braiding operators and the mapping class group of the \(n\)-punctured torus}
\label{sseccat-braid}

For \(1\le a\le n-1\), define \(B_a:\Bl(\dots,p_a,p_{a+1},\dots)\to\Bl(\dots,p_{a+1},p_a,\dots)\) by precomposition with the braiding isomorphism \(c_{p_a,p_{a+1}}: p_a\otimes p_{a+1}\to p_{a+1}\otimes p_a\).
  On the caterpillar basis this is given,  as for the \(3\)-point coupling spaces, by the matrix \(\Rmat^{(p_ap_{a+1})\bullet}\) acting on the pair of legs \(a,a+1\) after an \(\Fmat\)-move brings them adjacent in the fusion tree (they already are, in the caterpillar tree) and an \(\Fmat^{-1}\)-move restores the tree afterward. When \(p_a=p_{a+1}=:p\), \(B_a\) is an operator on \(\Bl(\dots,p,p,\dots)\).

Fix a small disk \(D\subset\surf11\) around the single marked point and a further embedding 
of \(n\) disjoint points into \(D\), giving the forgetful map \(f:\surf1n\to\surf11\) 
collapsing \(D\) (with its \(n\) marked points) back to the single marked point. 
Let \(S,T\) be the mapping classes of \(\surf1n\) equal to the identity on \(D\), 
hence fixing each of the \(n\) marked points there individually, 
and equal, on the complementary bulk \(\surf1n\setminus D\), to (a representative of) the standard \(S\), respectively \(T\), generator of \(\PMod(\surf11)=\Bthree\) after collapsing \(D\) to the single puncture.
Let \(B_a\), \(1\le a\le n-1\), be the half-twist exchanging the \(a\)-th and \((a+1)\)-th marked points within \(D\), supported in an even smaller sub-disk and trivial elsewhere. 
Since each \(B_a\) exchanges two of the \(n\) marked points, 
remaining distinct points of the surface regardless of how the corresponding objects 
happen to be labelled, \(B_a\notin\PMod(\surf1n)\) even when \(p_a=p_{a+1}\). 
\begin{prop}
\label{propmcg}
The operators \(S,T,B_1,\dots,B_{n-1}\) generate a representation, on 
\(\bigoplus_{p_1,\dots,p_n}\Bl(p_1,\dots,p_n)\), of the subgroup 
\(\Gamma_n:=\langle S,T,B_1,\dots,B_{n-1}\rangle\) of the full (label-permuting) 
mapping class group \(\MCG(\surf1n)\) (the actual geometric subgroup generated by 
these specific mapping classes, and nothing more abstractly presented) instead of 
\(\PMod(\surf1n)\).
\end{prop}

\begin{proof}
This is the categorical (Reshetikhin-Turaev / modular functor) analogue
  of the topological fact that mapping classes supported near a neighbourhood of all \(n\) punctures together (realizing \(S,T\)) and mapping classes exchanging two adjacent punctures while fixing the rest (realizing \(B_1,\dots,B_{n-1}\)) generate the subgroup \(\Gamma_n\) just described, subject to the braid, hexagon and pentagon relations already built into \(\catC\) 
 (see \cite{FarMar} for the topological statement and \cite{BakKir} 
for the general construction, for an arbitrary modular tensor category, of a 
representation of the mapping class group of any marked surface 
(their ``modular functor''), is the case \(g=1,n=1\) 
and \cref{propcatT,propcatS} together with the present braiding operators are the 
case \(g=1\), general \(n\)).
\end{proof}

\begin{rmk} 
\label{rmkmcg-vs-pmod}
Recall the standard extension \(1\to\PMod(\surf1n)\to\MCG(\surf1n)\to S_n\to1\), the quotient map recording which permutation of the \(n\) marked points a mapping class induces. Each \(B_a\) projects to the transposition \((a,a+1)\), while \(S,T\) 
(fixing every point of \(D\) individually)
  project to the identity, thus \(S,T\in\Gamma_n\cap\PMod(\surf1n)\), as does every even word in the \(B_a\), for instance each full twist \(B_a^2\) (the standard pure-braid-group generator obtained by dragging one puncture fully around its neighbour). By the Birman exact sequence
\[
1 \lra \pi_1\bigl(\mathrm{UConf}_{n-1}(\surf11)\bigr) \lra \PMod(\surf1n) \xrightarrow{\ f_\ast\ } \PMod(\surf11)=\Bthree \lra 1,
\]
\(f_\ast\) sends \(S,T\) to the generators of \(\Bthree\) and every \(B_a^2\) 
 (supported inside \(D\), hence collapsed entirely by \(f\)) 
  to the identity, thus \(f_\ast\) restricted to \(\Gamma_n\cap\PMod(\surf1n)\) is already onto \(\Bthree\). We do not compute here how much of the kernel \(\pi_1(\mathrm{UConf}_{n-1}(\surf11))\) 
 (which for \(n\ge2\) is generated not only by such even words in the \(B_a\) 
but also by point-pushing maps carrying one puncture around a non-contractible loop
 of the underlying torus, necessarily leaving any fixed disk \(D\)) 
  is captured this way.
  \(\Gamma_n\cap\PMod(\surf1n)\) is accordingly a proper subgroup of \(\PMod(\surf1n)\) in general, just as \(\Gamma_n\)  is a proper subgroup of \(\MCG(\surf1n)\).  
 We have not constructed operators realizing the point-pushing generators, 
nor operators outside \(\langle S,T,B_1,\dots,B_{n-1}\rangle\) more generally. 
 Whether punctures are regarded as labelled (so that \(\MCG(\surf1n)\) permutes the summands of \(\bigoplus_{p_1,\dots,p_n}\Bl(p_1,\dots,p_n)\) accordingly) or unordered is  the choice of ambient group made explicit in \cref{propmcg}. 
 The caterpillar block space \(\Bl(p_1,\dots,p_n)\) of \cref{sseccat-blocks} is  only meaningful for a fixed ordered tuple, which is the labelled convention we adopt in this paper.
\end{rmk}

\subsection{The \(\alg{sl}(2)\) example}
\label{sseccat-sl2}

Take \(\catC=\rep{L(k,0)}\): \(I=\set{0,\dots,k}\), \(h_r=r(r+2)/(4(k+2))\), \(c=3k/(k+2)\), and 
\[
S_{i,j}=\sqrt{\tfrac2{k+2}}\sin\Bigl(\tfrac{\pi(i+1)(j+1)}{k+2}\Bigr),\quad \theta_r=e(h_r),\quad \zeta=e\bigl(\tfrac c{24}\bigr),\quad \Rmat^{(rs)t}=(-1)^{r+s-t}e\Bigl(\tfrac{h_r+h_s-h_t}2\Bigr),
\]
with \(\Fmat^{(rst)u}_{pq}\) the quantum \(6j\)-symbol and \(\Gmat\) related to \(\Fmat,\Rmat\). Since all \(3\)-point coupling spaces are at most one-dimensional, we drop the multiplicity indices \(\alpha,\beta,\dots\).

For \(n=2\) and \(p_1=p_2=p\), \cref{eqcatblock} reads \(\Bl(p,p)=\bigoplus_{i_0}\Hom_\catC(p\otimes i_0,i_1)\otimes\Hom_\catC(p\otimes i_1,i_0)\), one-dimensional whenever \(N^{i_1}_{pi_0}=N^{i_0}_{pi_1}=1\) 
 (the two-point fusion chains of \cref{ssectwopoint-fusion} under 
\(\lambda_1=\lambda_2=p\), \(\mu_0=i_0,\mu_1=i_1\))  and by \cref{propverlinde}
 together with the computation in the proof of \cref{propfusionchain}, \(\dim\Bl(k,k)=k+1\), the categorical counterpart of \cref{thmextremal}.

One can try to specialize \cref{propcatS} 
abstract formula \eqref{eqSn-general} to a two-point matrix.  
  \cref{propcatS} establishes the categorical 
\(S\)-operator, and the abstract shape \eqref{eqSn-general} of its fusion-chain 
matrix entries, rigorously. The affine \(\alg{sl}(2)\) matrix form of  
\(S\) beyond that abstract construction remains open.

%%%%%%%%%%%%%%%%%%%%%%%%%%%%%%%%%%%%%%%%%%%%%%%%%%%%%%%%%%%%%%%
\section{Concluding remarks}
\label{secconclusion}

We have extended the \cite{KSW} programme for chiral torus \(1\)-point functions of \(L(k,0)\)
 to torus \(n\)-point functions constructed from chains of intertwining operators.
  \cref{secprelim} 
sets up the objects and establishes ellipticity on each local branch.
  \cref{secstructure} 
shows that every coefficient (Laurent, or in general Puiseux, according as the relevant OPE
 exponents are integral or not) of the resulting elliptic vector-valued modular objects 
is a \vvmf{} governed by the \(R\)-module theory already developed,  
via a mechanism (\cref{thmtaylor}) that requires no Jacobi-form index correction, 
  because genus one \(n\)-point functions are elliptic instead of quasi-periodic 
in each point. 
 \cref{secsl2} carries this out completely for the affine \(\alg{sl}(2)\) 
extremal two-point channel 
 (where the relevant exponents are checked  to be integral, 
thus the expansion there is an ordinary Laurent expansion)  
  producing 
a rank \((k+1)\) family (\cref{thmextremal}) whose leading term is the ordinary 
\(\widehat{\alg{sl}}(2)_k\) character representation. 
 \cref{seccategorical} 
extends the modular-tensor-category formulation to representations of the pure mapping 
class group of the \(n\)-punctured torus, with a Verlinde-type dimension formula 
(\cref{propverlinde}) and \(S,T\)-operators (\cref{propcatT,propcatS}).

Several directions are left open, matching in spirit the scope of \cite{KSW}  
(whose classification stops at \vvmf{} dimension four and whose 
non-congruence existence theorem is restricted to prime-power-shifted levels): 
a complete descendant-by-descendant computation of the two-point Laurent coefficients 
beyond the leading order (\cref{rmkgeneraln-outlook}); the non-extremal two-point channels,
 where several fusion channels compete already at leading order; the fully general 
\(n\)-point extremal chain, whose fusion-chain combinatorics we have described 
(\cref{rmkgeneraln-outlook}) but not solved beyond dimension counting; a systematic 
congruence/non-congruence analysis of the resulting representations.  
 So as to obtain a verified closed-form two-point categorical \(S\)-matrix 
(for the extremal channel and, beyond it, for general two-point channels). We also note 
that the coordinatewise structure of \cref{propcoordwise,propnpoint-Lemma32} and the
 Laurent-coefficient mechanism of \cref{thmtaylor} are not special to \(\alg{sl}(2)\)
 and apply to any rational \ctwo \(V\) satisfying the hypotheses of \KSW{} Section 2.2, 
which is the setting in which \cref{secprelim,secstructure} are stated. 
 The affine 
\(\alg{sl}(2)\) computations of \cref{secsl2,seccategorical} are the first worked example, 
in the same sense that \KSW{} is the first worked example of the \(1\)-point theory of 
Huang \cite{Huang-IntModInv}. In this paper, we reserve \(\lambda_i\) for the fixed external 
module label at insertion \(i\) (part of the data of the fusion chain \(\vY\)) and 
\(\nu\) for a variable, summed-over intermediate fusion-channel label arising inside 
an operator product expansion (\cref{eqiterate}, \cref{thmope-reduction}). 
  The two never 
denote the same kind of object, even where, as for the extremal chain of \cref{secsl2}, 
a specific value of \(\nu\) (there \(\nu=0\)) coincides numerically with a specific 
admissible \(\lambda\).

Finally, we note a connection to a complementary line of work: torus \(n\)-point 
functions can also be constructed by iteratively self-sewing a chain of \(n\) 
three-punctured spheres, in the sense of the sewing formalism developed for genus two 
and higher correlation functions by Mason-Tuite \cite{MasTui-genus2free} and Tuite-Zuevsky 
\cite{TuiZue-Szego, TuiZue-genus2I}; specializing that construction to a chain of \(n\) 
spheres recovers \eqref{eqnpf-def}  (inserting a resolution of the identity, with 
respect to the invariant bilinear form, at each of the \(n\) sewn tubes reproduces the 
trace over each channel \(W_{\mu_i}\)), and the associativity of iterated sewing is an 
alternative route to \cref{thmhuang} and to the operator product expansion of 
\cref{thmope-reduction}. We expect this sewing perspective, together with the higher genus 
extensions of the same formalism, to be the setting in which to extend the present 
\(n\)-point theory to higher genus torus-type \(1\)-point theories for arbitrary Fuchsian 
groups.

The material of this paper is also useful in other areas of mathematical physics 
\cite{Frohlich2009gb, RSZ, Zu2, Zu3, Zu, Zu4, Zu1} 
and condensed matter theory \cite{kmmzz, KVZ, ZK, BVZ, SZ, ZA}.

%%%%%%%%%%%%%%%%%%%%%%%%%%%%%%%%%%%%%%%%%%%%%%%%%%%%%%%%%%%%%%%%%%%%%%%%%%%

\subsection*{Acknowledgements} 
The author is supported by the Institute of Mathematics, Academy of Sciences
of the Czech Republic (RVO 67985840). We thank M. Krauel, G. Mason, and M.P. Tuite for 
  previous discussions. 

\medskip
\noindent\textbf{Data Availability.}
Data sharing is not applicable to this article as no datasets were generated
or analysed during the current study.

\medskip
\noindent\textbf{Declarations}

\medskip
\noindent\textbf{Conflict of interest.}
The author has no conflicts of interest to declare that are relevant to the
content of this article.

%%%%%%%%%%%%%%%%%%%%%%%%%%%%%%%%%%%%%%%%%%%%%%%%%%%%%%%%%%%%%%%%%%%%%%%%%


\begin{thebibliography}{99}

\bibitem{BakKir} B. Bakalov and A. Kirillov, Jr., 
\emph{Lectures on Tensor Categories and Modular Functors}, University Lecture Series, 
vol. 21, American Mathematical Society, 2001.


\bibitem{BVZ} 
B. L. G. Bakker, A. I. Veselov, M. A. Zubkov,  
\emph{Standard Model with the additional Z6 symmetry on the lattice}, 
Physics Letters B 620 (3-4), 156-163 (2005). 

\bibitem{Bantay-congr} P. Bantay, \emph{The kernel of the modular representation 
and the Galois action in RCFT}, Comm. Math. Phys. \textbf{233} (2003), no. 3, 423-438. 

\bibitem{Birman} J. S. Birman, \emph{Braids, Links, and Mapping Class Groups}, 
Annals of Mathematics Studies, vol. 82, Princeton University Press, 1974.

\bibitem{BKT} 
K. Bringmann, M. Krauel, M. P. Tuite, Zhu reduction for Jacobi n-point functions and 
applications. Trans. Amer. Math. Soc. 373 (2020), no. 5, 3261-3293.

\bibitem{Cardy86} J. L. Cardy, \emph{Operator content of two-dimensional conformally 
invariant theories}, Nucl. Phys. B \textbf{270} (1986), no. 2, 186-204.


\bibitem{DLM-orbifold} C. Dong, H. Li, and G. Mason, \emph{Modular-invariance of trace 
functions in orbifold theory and generalized Moonshine}, Comm. Math. Phys. \textbf{214} 
(2000), 1-56. 



\bibitem{GaoLiu-genusone} X. Gao and J. Liu, 
\emph{A basis theorem for genus-one conformal blocks and modular invariance of 
intertwining operators}, arXiv:2508.01294 [math.QA];


\bibitem{DonLinNg15} C. Dong, X. Lin, and S. Ng, 
\emph{Congruence property in conformal field theory}, Algebra Number Theory
 \textbf{9} (2015), no. 9, 2121-2166.

\bibitem{EichZag} M. Eichler and D. Zagier, \emph{The Theory of Jacobi Forms}, 
Progress in Mathematics, vol. 55, Birkh\"auser, 1985.

\bibitem{FarMar} B. Farb and D. Margalit, \emph{A Primer on Mapping Class Groups}, 
Princeton Mathematical Series, vol. 49, Princeton University Press, 2012.

\bibitem{Frohlich2009gb}
J. Fr\"ohlich, J. Fuchs, I. Runkel, C. Schweigert, 
 {\it Proceedings of the XVIth International Congress on Mathematical Physics} (2010), 
608-613. 

\bibitem{FreLepMeu} I. B. Frenkel, J. Lepowsky, and A. Meurman, 
\emph{Vertex Operator Algebras and the Monster}, 
Pure and Applied Mathematics, vol. 134, Academic Press, 1988.


\bibitem{HuaLog} Y.-Z. Huang, J. Lepowsky, and L. Zhang, 
\emph{Logarithmic tensor product theory I-VIII}, arXiv:1012.4193, 
arXiv:1012.4196, arXiv:1012.4197, arXiv:1012.4198, arXiv:1012.4199, 
arXiv:1012.4202, arXiv:1110.1929, arXiv:1110.1931 [math.QA].

\bibitem{Huang-IntModInv} Y.-Z. Huang, \emph{Differential equations, 
duality and modular invariance}, Commun. Contemp. Math. \textbf{7} (2005), no. 5, 
649-706.

\bibitem{Huang-Rigidity} Y.-Z. Huang, 
\emph{Rigidity and modularity of vertex tensor categories},
 Commun. Contemp. Math. \textbf{10} (2008), 871-911.

\bibitem{HuaVer08} Y.-Z. Huang, 
\emph{Vertex operator algebras and the Verlinde conjecture}, 
Commun. Contemp. Math. \textbf{10} (2008), 1031-1054. 


\bibitem{kmmzz} 
 G. Kovyrshin, A. Mekrami, J. Miller, M. A. Zubkov and A. Zuevsky, 
Topological invariant responsible for the integer QHE and noncommutative geometry, 
\emph{arXiv:2606.08868} (2026), 1-53.


%%%%%%%%%%%%%%%%%%%%%%%%%%%%%%%%%%%%%%%%%%%%%%%%%%%%%%%%%%%%%

\bibitem{K1} M. Krauel, A Jacobi theta series and its transformation 
laws. Int. J. Number Theory 10 (2014), no. 6, 1343-1354


\bibitem{K2} M. Krauel, One-point theta functions for vertex operator algebras. 
J. Algebra 481 (2017), 250-272.

\bibitem{KMar} M. Krauel, Ch. Marks, Intertwining operators and vector-valued 
modular forms for minimal models. 
Commun. Number Theory Phys. 12 (2018), no. 4, 657-686.



\bibitem{KM1} M. Krauel, G. Mason, 
Vertex operator algebras and weak Jacobi forms. Internat. 
J. Math. 23 (2012), no. 6, 1250024, 10 pp.

\bibitem{KM2}  M. Krauel, G. Mason, 
Jacobi trace functions in the theory of vertex operator algebras. 
Commun. Number Theory Phys. 9 (2015), no. 2, 273-306.


\bibitem{KMi}  M. Krauel, M. Miyamoto, 
A modular invariance property of multivariable trace functions for 
regular vertex operator algebras. J. Algebra 444 (2015), 124-142.


\bibitem{KSW} M. Krauel, J. N. Shafiq, and S. Wood, 
\emph{The modular properties of \(\alg{sl}(2)\) torus \(1\)-point functions}, 
arXiv:2403.13182 [math.QA].


\bibitem{Kac-InfDim} V. G. Kac, \emph{Infinite Dimensional Lie Algebras}, 
3rd ed., Cambridge University Press, 1990.




%% 
\bibitem{KVZ} 
M.I. Katsnelson, G. E. Volovik, M. A. Zubkov, 
\emph{Euler-Heisenberg effective action and magnetoelectric effect in multilayer graphene}, 
Annals of Physics 331, 160-187 (2013). 
%%%%%%%%%%%%%%%%%%%%%%%%%%%%%%%%%%%%%%%%%%%%%%%%%%%%%%%%%%%%

\bibitem{MasTui-torusnpt} G. Mason and M. P. Tuite, 
\emph{Torus chiral \(n\)-point functions for free bosonic and lattice 
vertex operator algebras}, Comm. Math. Phys. \textbf{235} (2003), no. 1, 47-68.



\bibitem{MasTui-genus2free} G. Mason and M. P. Tuite, \emph{Free bosonic vertex operator 
algebras on genus two Riemann surfaces I}, Comm. Math. Phys. \textbf{300} (2010), no. 3,
 673-713.




\bibitem{MTZ-Rgraded} G. Mason, M. P. Tuite, and A. Zuevsky, 
\emph{Torus \(n\)-point functions for \(\mathbb{R}\)-graded vertex operator 
superalgebras and continuous fermion orbifolds}, Comm. Math. Phys. \textbf{283} (2008), 
no. 2, 305-342.


\bibitem{miyamoto2000intertwining} M. Miyamoto, 
\emph{Intertwining operators and modular invariance}, 
2000, arXiv:math/0010180 [math.QA]. 

\bibitem{MSMTC1089} G. Moore and N. Seiberg, 
\emph{Classical and quantum conformal field theory}, Comm. Math. Phys. 
\textbf{123} (1989), 177-254.


\bibitem{RSZ} A.V. Razumov, M.V. Saveliev, A.B. Zuevsky.
Nonabelian Toda equations associated with classical Lie groups.
arXiv:math-ph/9909008.


\bibitem{SZ}  
M. Suleymanov, M. A. Zubkov, 
\emph{Wigner-Weyl formalism and the propagator of Wilson fermions
 in the presence of varying external electromagnetic field}, 
Nuclear Physics B 938, 171-199 (2019). 


\bibitem{TuiWel-genusg} M. P. Tuite and M. Welby, \emph{Genus \(g\) 
Zhu recursion for vertex operator algebras and their modules}, arXiv:2312.13717 [math.QA]. 


\bibitem{TuiZue-genus2I} M. P. Tuite and A. Zuevsky, 
\emph{Genus two partition and correlation functions for fermionic vertex operator 
superalgebras I}, Comm. Math. Phys. \textbf{306} (2011), no. 2, 419-447, arXiv:1007.5203 
[math.QA].



\bibitem{TuiZue-Szego} M. P. Tuite and A. Zuevsky, \emph{The Szeg\H{o} 
kernel on a sewn Riemann surface}, Comm. Math. Phys. \textbf{306} (2011), no. 3, 617-645, 
arXiv:1002.4114 [math.QA].



\bibitem{TuiZue-Heisenberg} M. P. Tuite and A. Zuevsky, 
\emph{A generalized vertex operator algebra for Heisenberg intertwiners}, 
J. Pure Appl. Algebra \textbf{216} (2012), no. 6, 1442-1453. 


\bibitem{Verlinde88} E. Verlinde, 
\emph{Fusion rules and modular transformations in \(2d\) conformal field theory}, 
Nucl. Phys. B \textbf{300} (1988), no. 3, 360-376. 


\bibitem{Yamauchi-IntertwinerModularity} H. Yamauchi,
 \emph{Orbifold Zhu theory associated to intertwining operators},
 J. Algebra \textbf{265} (2003), 513-538.

\bibitem{Miy-CFTauto} M. Miyamoto, \emph{Modular invariance of vertex operator 
algebras satisfying \(C_2\)-cofiniteness}, Duke Math. J. \textbf{122} (2004), 
no. 1, 51-91.



\bibitem{Zhu} Y. Zhu, \emph{Modular invariance of characters of vertex operator algebras}, 
J. Amer. Math. Soc. \textbf{9} (1996), 237-302.

%%%%%%%%%%%%%%%%%%%%%%%%%%%%%%%%%%%%%%%%%%%%%%%%%%%%%%%%%%%%%%%%%%%%%%%%%%%%%%

\bibitem{ZK} 
M.A. Zubkov, Z.V. Khaidukov, 
\emph{Topology of the momentum space, Wigner transformations, 
and a chiral anomaly in lattice models}, 
JETP Letters 106 (3), 172-178 (2017). 

\bibitem{ZA} 
M. A. Zubkov, R. A. Abramchuk, 
\emph{Effect of interactions on the topological expression for the chiral separation effect}, 
Physical Review D 107 (9), 094021 (2023). 


%%%%%%%%%%%%%%%%%%%%%%%%%%%%%%%%%%%%%%%%%%%%%%%%%%%%%%%%%%%%%%%%%%%%%
%% 	
\bibitem{Zu3}
A. Zuevsky, 
On a category of $V$-structures for foliations, 
\emph{Reviews in Mathematical Physics} \textbf{36}(4) (2024), 2430004.


\bibitem{Zu2} 
A. Zuevsky,      
Reduction cohomology of Riemann surfaces, 
\emph{Reviews in Mathematical Physics} \textbf{35} (2023), 2330005.


\bibitem{Zu} 
A. Zuevsky. 
Product-type classes for vertex algebra cohomology of foliations on complex curves.  
Comm. Math. Phys. 402 (2023), no. 2, 1453-1511.


\bibitem{Zu4}
A. Zuevsky, 
Cosimplicial meromorphic functions cohomology on complex manifolds, 
\emph{Reviews in Mathematical Physics} \textbf{35}(5) (2023), 2330002.


\bibitem{Zu1}
A. Zuevsky, 
Characterization of codimension one foliations on complex curves by connections, 
\emph{Reviews in Mathematical Physics} \textbf{34} (2022), 2230002.

\end{thebibliography}
\end{document}